\documentclass{birkjour_t2}
 \newtheorem{thm}{Theorem}[section]
 
 \newtheorem{lem}[thm]{Lemma}
 
 \theoremstyle{definition}
 
 \theoremstyle{remark}

 \numberwithin{equation}{section}

\usepackage{verbatim}
\usepackage{xcolor}
\usepackage{graphicx}
\usepackage{amssymb, amsmath, amsthm}
\usepackage{amsfonts}
\usepackage{amssymb}
\usepackage{float}
\usepackage{hyperref}
\usepackage{mathrsfs}
\usepackage{enumitem}
\newtheorem{theorem}{\bf Theorem}

\newtheorem{case}{\bf Case}

\newtheorem{lemma}[theorem]{\bf Lemma}

\newtheorem{proposition}[theorem]{\bf Proposition}
\newtheorem{remark}[theorem]{\bf Remark}

\newtheorem{step}{\bf Step}

\newcommand{\sign}{{\rm sgn}}
\newcommand{\dis}{\displaystyle}

\newcommand{\N}{\mathbb N}
\newcommand{\Z}{\mathbb Z}

\newcommand{\R}{\mathbb R}

\newcommand{\intox}{\int_{\mathbb{T}^d}}

\renewcommand{\theequation}{\thesection.\arabic{equation}}
\numberwithin{equation}{section}
\numberwithin{theorem}{section}

\numberwithin{figure}{section}

\begin{document}

%
%
%
%
%
%
%
%
%

\title[Wellposedness and Convergence of solutions]
 {Wellposedness and convergence of solutions to a class of forced non-diffusive equations with applications}

\author{Susan Friedlander}

\address{Department of Mathematics\\
University of Southern California}

\email{susanfri@usc.edu}

\author{Anthony Suen}

\address{Department of Mathematics and Information Technology\\
The Education University of Hong Kong}

\email{acksuen@eduhk.hk}

\date{}



\begin{abstract}
This paper considers a family of non-diffusive active scalar equations where a viscosity type
parameter enters the equations via the constitutive law that relates the drift velocity with the
scalar field. The resulting operator is smooth when the viscosity is present but singular
when the viscosity is zero. We obtain Gevrey-class local well-posedness results and convergence
of solutions as the viscosity vanishes. We apply our results to two examples that are derived from
physical systems: firstly a model for magnetostrophic turbulence in the Earth's fluid core and
secondly flow in a porous media with an ``effective viscosity''.
\end{abstract}

\maketitle
\section{Introduction}\label{introduction}

Active scalar equations arise in many areas of fluid dynamics, with the most classical being the two dimensional Euler equation for an incompressible, inviscid flow in vorticity form.  Another much studied active scalar equation is the surface quasi-geostrophic equation (SQG) introduced by Constantin, Majda and Tabak \cite{CMT94} as a two dimensional analogue of the three dimensional Euler equation (c.f \cite{CW09}, \cite{CV10}, \cite{FFWV12}, \cite{KNV07}, \cite{R95}). The physics of an active scalar equation is encoded in the constitutive law that relates the transport velocity vector $u$ with the scalar field $\theta$. This law produces a differential operator that when applied to the scalar field determines the velocity.  The singular or smoothing properties of the operator are closely connected with the mathematical properties of the active scalar equation. In this present paper we study the following class of non-diffusive active scalar equations in $\mathbb{T}^d\times(0,\infty)=[0,2\pi]^d\times(0,\infty)$ with $d\ge2$:
\begin{align}
\label{abstract active scalar eqn introduction} \left\{ \begin{array}{l}
\partial_t\theta^{\nu}+u^\nu\cdot\nabla\theta^{\nu}=S, \\
u_j^{\nu}=\partial_{x_i} T_{ij}^{\nu}[\theta^{\nu}],\theta^{\nu}(x,0)=\theta_0(x)
\end{array}\right.
\end{align}
where $\nu\ge0$. Here $\theta_0$ is the initial condition and $S=S(x)$ is a given smooth function that represents the forcing of the system.

Our motivation for addressing such a class of active scalar equations comes from two rather different physical systems that under particular parameter regimes give rise to systems of the form \eqref{abstract active scalar eqn introduction}. The first example comes from MHD and a model proposed by Moffatt and Loper \cite{ML94}, Moffatt \cite{M78} for magenetostrophic turbulence in the Earth's fluid core. Under the postulates in \cite{ML94}, the governing equation reduces to a 3 dimensional active scalar equation for a temperature field $\theta$
\begin{align}\label{active scalar equation general MG}
\partial_t \theta + u \cdot \nabla \theta = \kappa \Delta \theta + S
\end{align}
where the constitutive law is obtained from the linear system
\begin{align}
e_3 \times u &= - \nabla P + e_2 \cdot \nabla b + \theta e_3 + \nu \Delta u,\label{MG equation linear system 1}\\
0 &= e_2 \cdot \nabla u + \Delta b, \label{MG equation linear system 2}\\
\nabla \cdot u &= 0, \nabla \cdot b = 0. \label{MG equation linear system 3}
\end{align}
This system encodes the vestiges of the physics in the problem, namely the Coriolis force, the Lorentz force and gravity. Vector manipulations of \eqref{MG equation linear system 1}-\eqref{MG equation linear system 3} give the expression
\begin{align}\label{Vector manipulations}
\{[\nu \Delta^2 - (e_2 \cdot \nabla)^2]^2 + (e_3 \cdot \nabla)^2 \Delta \} u &= - [\nu \Delta^2 - (e_2 \cdot \nabla)^2] \nabla \times (e_3 \times \nabla \theta)\notag\\
&\qquad + (e_3 \cdot \nabla)\Delta (e_3 \times \nabla \theta).
\end{align}
Here $(e_1, e_2, e_3)$ denote Cartesian unit vectors. The explicit expression for the components of the Fourier multiplier symbol $\widehat{M}^\nu$ as functions of the Fourier variable $k = (k_1, k_2, k_3) \in \mathbb{Z}^3$ are obtained from the constitutive law \eqref{Vector manipulations} to give
\begin{align}
\widehat M^{\nu}_1(k)&=[k_2k_3|k|^2-k_1k_3(k_2^2+\nu|k|^4)]D(k)^{-1},\label{MG Fourier symbol_1}\\
\widehat M^{\nu}_2(k)&=[-k_1k_3|k|^2-k_2k_3(k_2^2+\nu|k|^4)]D(k)^{-1},\label{MG Fourier symbol_2}\\
\widehat M^{\nu}_3(k)&=[(k_1^2+k_2^2)(k_2^2+\nu|k|^4)]D(k)^{-1},\label{MG Fourier symbol_3}
\end{align}
where
\begin{align}
D(k)=|k|^2k_3^2+(k_2^2+\nu|k|^4)^2.\label{MG Fourier symbol_4}
\end{align}
The nonlinear equation \eqref{active scalar equation general MG} with $u$ related to $\theta$ via \eqref{Vector manipulations} is called the magnetogeostrophic (MG) equation and its mathematical properties have been studied in a series of papers including \cite{FRV12}, \cite{FRV14}, \cite{FS15}, \cite{FS18}, \cite{FV11a}, \cite{FV11b}.
In the magnetostrophic turbulence model the parameters $\nu$, the nondimensional viscosity, and $\kappa$, the nondimensional thermal diffusivity, are extremely small. The behavior of the MG equation is dramatically different when the parameters $\nu$ and $\kappa$ are present (i.e. positive) or absent (i.e. zero). The limit as either or both parameters vanish in highly singular. Since both parameters multiply a Laplacian term, their presence is smoothing. However $\kappa$ enters \eqref{active scalar equation general MG} in a parabolic heat equation role whereas $\nu$ enters via the constitutive law \eqref{Vector manipulations}.  The mathematical properties of the MG equation have been determined in various settings of the parameters via an analysis of the Fourier multiplier symbol $\widehat{M}^\nu$ given by \eqref{MG Fourier symbol_1}-\eqref{MG Fourier symbol_4}.

We note that the Fourier multiplier symbols $\widehat{M}^0$ given by \eqref{MG Fourier symbol_1}-\eqref{MG Fourier symbol_4} with $\nu = 0$ are {\it not} bounded in all regions of Fourier space \cite{FV11a}. More specifically in ``curved" regions where $k_3 = \mathcal{O}(1), \; k_2 = \mathcal{O}(|k_1|^{1/2})$ the symbols are unbounded as $|k_1| \to \infty$ with $|\widehat{M}^0(k)| \le C|k|$ for some positive constant $C$.  Thus when $\nu = 0$ the relation between $u$ and $\theta$ is given by a {\it singular} operator of order 1.  The implications of this fact for the inviscid MG equation are summarized in the survey article by Friedlander, Rusin and Vicol \cite{FRV12}.  In particular, when $\kappa > 0$ the inviscid but thermally dissipative MG equation is globally well-possed. In contrast when $\nu = 0$ {\it and} $\kappa = 0$, the singular inviscid MG$^0$ equation is {\it ill-possed} in the sense of Hadamard in any Sobolev space. In a recent paper \cite{FS18} Friedlander and Suen examine the limit of vanishing viscosity in the case when $\kappa > 0$. They prove global existence of classical solutions to the forced MG$^\nu$ equations and obtain strong convergence of solutions as the viscosity vanishes.  In this present paper we turn to the case where $\kappa = 0$ and examine the MG$^\nu$ system, without the benefit of thermal diffusion, both when $\nu > 0$ and in the case $\nu=0$ where the operator MG$^0$ is singular of order 1.  We obtain Gevrey-class local well-posedness and convergence of solutions as $\nu \to 0$.  The precise statements of the theorems are given in Section~\ref{main results} in the context of a general class of non-diffusive active scalar equations that includes the MG equations with $\kappa=0$.

The second example of a physical system which can be modeled by an active scalar equation where a small smoothing parameter enters into the constitutive law comes from flow in a porous medium.  The incompressible porous media Brinkman equation with an ``effective viscosity" $\nu$ is derived via a modified Darcy's Law as suggested by Brinkman \cite{B49}.  The 2D equation relating the velocity $u$, the density $\theta$ and the pressure $P$ is given in non-dimensional form by
\begin{align}
u &= - \nabla P - e_2 \theta + \nu \Delta u \label{IPMB equation linear system 1}\\
\nabla \cdot u &= 0 \label{IPMB equation linear system 2}
\end{align}
which produces the constitutive law
\begin{align}\label{IPMB constitutive law}
u &= (1 - \nu \Delta)^{-1} [-\nabla \cdot (-\Delta)^{-1} e_2 \cdot \nabla \theta - e_2\theta]\notag\\
&= (1-\nu \Delta)^{-1} R^{\bot} R_1 \theta
\end{align}
where $R = (R_1, R_2)$ is the vector of Riesz transforms. The 2D components of the Fourier multiplier symbol corresponding to \eqref{IPMB constitutive law} are
\begin{equation}\label{Fourier multiplier symbol IPMB}
\frac{1}{1 + \nu (k^2_1 + k^2_2)} \left(\frac{k_1k_2}{k^2_1 + k^2_2} \quad , \quad \frac{-k^2_1}{(k^2_1 + k^2_2)}\right)
\end{equation}
Again there is a dramatic difference in the operator between the two cases $\nu > 0$ and $\nu = 0$. In the first case the operator is smoothing of order 2 and in the second case the operator is singular of order zero. The IPMB active scalar example is a 2 dimensional nonlinear equation for $\theta$ given by
\begin{equation}\label{active scalar equation general IPMB}
\partial_t \theta + (u \cdot \nabla) \theta = 0
\end{equation}
coupled with the constitutive law \eqref{IPMB constitutive law}.

The well known IPM equations, i.e. \eqref{IPMB constitutive law}-\eqref{active scalar equation general IPMB} without the effective viscosity $\nu$, have been studied in a number of papers, c.f \cite{CFG11}, \cite{CGO07}. As we observed when $\nu = 0$ the operator in \eqref{IPMB constitutive law} is a singular integral operator of order zero. This is also the case for the SQG equations. However the SQG operator is odd where as the IPM operator is even, a property that it shares with the MG operator.  Implications for well/ill posedness due to the odd/even structure of the operator in an active scalar equation are explored in \cite{FFWV12}, \cite{FRV14}, \cite{KVW16}. In this present paper we study the system \eqref{IPMB constitutive law}-\eqref{active scalar equation general IPMB} in the limit of vanishing viscosity.  We obtain results analogous to those for the MG$^\nu$ system in the limit of vanishing viscosity.  The principal difference is that the MG$^0$ operator is singular of order 1 where as the IPM operator is singular of order zero.  In the ``smoother" IPM case our convergence results are valid in Sobolev spaces rather than the Gevrey-class convergence results for the MG equation.
\section{Main Results for a General Class of Active Scalar Equations}\label{main results}
We now return to the abstract formulation of our problem in the setting of the following system of active scalar equations parameterized by a ``viscosity" parameter $\nu$.
\begin{align}
\label{abstract active scalar eqn} \left\{ \begin{array}{l}
\partial_t\theta^{\nu}+u^\nu\cdot\nabla\theta^{\nu}=S, \\
u_j^{\nu}=\partial_{x_i} T_{ij}^{\nu}[\theta^{\nu}],\theta^{\nu}(x,0)=\theta_0(x)
\end{array}\right.
\end{align}
where $\mathbb{T}^d\times(0,\infty)=[0,2\pi]^d\times(0,\infty)$ with $d\ge2$. We assume that $\int_{\mathbb{T}^d}\theta^\nu(x,t)dx=\int_{\mathbb{T}^d}S(x)=0$ for all $t\ge0$. $\{T_{ij}^{\nu}\}_{\nu\ge0}$ is a sequence of operators\footnote{For simplicity, we sometime write $T^0_{ij}=T_{ij}^{\nu}\Big|_{\nu=0}$ and $u^\nu=\partial_x T^\nu[\theta^\nu]$.} which satisfy:
\begin{enumerate}
\item[A1]  $\partial_i\partial_j T^{\nu}_{ij}f=0$ for any smooth functions $f$ for all $\nu\ge0$.
\item[A2] $T_{ij}^\nu:L^\infty(\mathbb{T}^d)\rightarrow BMO(\mathbb{T}^d)$ are bounded for all $\nu\ge0$.
\item[A3] For each $\nu>0$, there exists a constant $C_{\nu}>0$ such that for all $1\le i,j\le d$, $$|\widehat T^{\nu}_{ij}(k)|\le C_{\nu}|k|^{-3}, \forall k\in\mathbb{Z}^d.$$
\item[A4] For each $1\le i,j\le d$, $$\lim_{\nu\rightarrow0}\sum_{k\in\mathbb{Z}^d}|\widehat{T^{\nu}_{ij}}(k)-\widehat{T^0_{ij}}(k)|^2|\widehat{g}(k)|^2=0$$
for all $g\in L^2$.
\end{enumerate}
Moreover, we further assume that $\{T_{ij}^{\nu}\}_{\nu\ge0}$ satisfy either one of the following assumption:
\begin{enumerate}
\item[A5$_{1}$] There exists a constant $C_{0}>0$ independent of $\nu$, such that for all $1\le i,j\le d$, $$\sup_{\nu\in(0,1]}\sup_{\{k\in\mathbb{Z}^d\}}|\widehat T^{\nu}_{ij}(k)|\le C_{0};$$
$$\sup_{\{k\in\mathbb{Z}^d\}}|\widehat T^0_{ij}(k)|\le C_{0}.$$
\item[A5$_2$] There exists a constant $C_{0}>0$ independent of $\nu$, such that for all $1\le i,j\le d$, $$\sup_{\nu\in(0,1]}\sup_{\{k\in\mathbb{Z}^d\}}|k_i\widehat T^{\nu}_{ij}(k)|\le C_{0};$$
$$\sup_{\{k\in\mathbb{Z}^d\}}|k_i\widehat T^0_{ij}(k)|\le C_{0}.$$
\end{enumerate}
\noindent Here are some remarks regarding the assumptions on $T^{\nu}_{ij}$:
\begin{itemize}
\item The assumption A1 implies that $u^\nu$ is divergence-free for all $\nu\ge0$.
\item The assumption A3 is needed for obtaining global-in-time wellposedness for \eqref{abstract active scalar eqn} for the case $\nu>0$, which implies that $\partial_x T^\nu$ are operators of smoothing order 2 for $\nu>0$.
\item The assumption A5$_2$ is stronger than assumption A5$_1$. In particular, assumption A5$_2$ implies that $\partial_x T^\nu$ are operators of zero order.
\item The assumption A4 is needed for obtaining convergence of solutions as $\nu\rightarrow0$, and is consistent with the one given in \cite{FS18}.
\item All the assumptions A1--A4 and A5$_1$ are consistent with the case for the magnetogeostrophic (MG) equations, while assumptions A1--A4 and A5$_2$ are consistent with the case for incompressible porous media Brinkmann (IPMB) equations.
\end{itemize}

The main results that we prove in this present work are stated in the following theorems:

\begin{theorem}[Wellposedness in Sobolev space in the case $\nu>0$]\label{Wellposedness in Sobolev space}
Let $\theta_0\in W^{s,d}$ for $s\ge0$ and $S$ be a $C^\infty$-smooth source term. Then for each $\nu>0$, under the assumptions A1--A3 and A5$_i$ for $i=1$ or 2, we have:
\begin{itemize}
\item if $s=0$, there exists unique global weak solution to \eqref{abstract active scalar eqn} such that
\begin{align}
\theta^\nu&\in BC((0,\infty);L^d),\label{BC solution theta}\\
u^\nu&\in C((0,\infty);W^{2,d}).\label{Continuous u}
\end{align}
In particular, $\theta^\nu(\cdot,t)\rightarrow\theta_0$ weakly in $L^d$ as $t\rightarrow0^+$. Here $BC$ stands for {\it bounded continuous functions}.
\item if $s>0$, there exists a unique global-in-time solution $\theta^\nu$ to \eqref{abstract active scalar eqn} such that $\theta^\nu(\cdot,t)\in W^{s,d}$ for all $t\ge0$. Furthermore, for $s=1$, we have the following single exponential growth in time on $\|\nabla\theta^\nu(\cdot,t)\|_{L^d}$:
\begin{align}\label{exp growth}
\|\nabla\theta^\nu(\cdot,t)\|_{L^d}\le C \|\nabla\theta_0\|_{L^d}\exp\left(C\left(t\|\theta_0\|_{W^{1,d}}+t^2\|S\|_{L^\infty}+t\|S\|_{W^{1,d}}\right)\right),
\end{align}
where $C>0$ is a constant which depend only on $\nu$ and the spatial dimension $d$.
\end{itemize}
\end{theorem}

\begin{theorem}[Gevrey-class global wellposedness in the case $\nu>0$]\label{Gevrey-class global wellposedness}
Fix $s\ge1$. Let $\theta_0$ and $S$ be of Gevrey-class $s$ with radius of convergence $\tau_0>0$. Then for each $\nu>0$, under the assumptions A1--A3 and A5$_i$ for $i=1$ or 2, there exists a unique Gevrey-class $s$ solution $\theta^\nu$ to \eqref{abstract active scalar eqn} on $\mathbb{T}^d\times[0,\infty)$ with radius of convergence at least $\tau=\tau(t)$ for all $t\in[0,\infty)$, where $\tau$ is a decreasing function satisfying
\begin{align}\label{lower bound on tau}
\tau(t)\ge\tau_0e^{-C\left(\|e^{\tau_0\Lambda^\frac{1}{s}}\theta_0\|_{L^2}+2\|e^{\tau_0\Lambda^\frac{1}{s}}S\|_{L^2}\right)t}.
\end{align}
Here $C>0$ is a constant which depends on $\nu$ but independent of $t$.
\end{theorem}

\begin{theorem}[Gevrey-class local wellposedness in the case $\nu=0$]\label{Gevrey-class local wellposedness}
Fix $s\ge1$, $r>\frac{d}{2}+\frac{3}{2}$ and $K_0>0$. Let $\theta^0(\cdot,0)=\theta_0$ and $S$ be of Gevrey-class $s$ with radius of convergence $\tau_0>0$ and satisfy
\begin{align}\label{bounds on gevrey norm of S and theta0'}
\|\Lambda^re^{\tau_0\Lambda^\frac{1}{s}}\theta^0(\cdot,0)\|_{L^2}\le K_0,\qquad\|\Lambda^re^{\tau_0\Lambda^\frac{1}{s}}S\|_{L^2}\le K_0.
\end{align}
For $\nu=0$, under the assumptions A1--A2 and A5$_1$, there exists $\bar{T},\bar{\tau}>0$ and a unique Gevrey-class $s$ solution $\theta^0$ to \eqref{abstract active scalar eqn} defined on $\mathbb{T}^d\times[0,\bar{T}]$ with radius of convergence at least $\bar{\tau}$. In particular, there exists a constant $C=C(K_0)>0$ such that for all $t\in[0,\bar{T}]$,
\begin{align}
\|\Lambda^re^{\bar{\tau}\Lambda^\frac{1}{s}}\theta^0(\cdot,t)\|_{L^2}\le C.\label{bound on solution 2}
\end{align}
Moreover, if the assumption A3 holds as well, then we have
\begin{align}
\|\Lambda^re^{\bar{\tau}\Lambda^\frac{1}{s}}\theta^\nu(\cdot,t)\|_{L^2}\le C,\,\forall\nu>0,\label{bound on solution 1}
\end{align}
where $\theta^\nu$ are Gevrey-class $s$ solutions to \eqref{abstract active scalar eqn} for $\nu>0$ as described in Theorem~\ref{Gevrey-class global wellposedness}.
\end{theorem}

\begin{theorem}[Local wellposedness in Sobolev space in the case $\nu=0$ and Property A5$_2$ holds]\label{Local wellposedness in Sobolev space}
For $d\ge2$, we fix $s>\frac{d}{2}+1$. Assume that $\theta_0,S\in H^s(\mathbb{T}^d)$ have zero-mean on $\mathbb{T}^d$. Then for $\nu=0$, under the assumption A1--A2 and A5$_2$, there exists a $T>0$ and a unique smooth solution $\theta^0$ to \eqref{abstract active scalar eqn} such that
$$\theta^0\in L^\infty(0,T;H^s(\mathbb{T}^d)).$$
\end{theorem}

\begin{theorem}[Convergence of solutions as $\nu\rightarrow0$]\label{Convergence of solutions as nu goes to 0} Depending on the assumptions A$5_1$ and A$5_2$, we have the following cases:
\begin{itemize}
\item Assume that the hypotheses and notations of Theorem~\ref{Gevrey-class local wellposedness} are in force. Under the assumptions A3--A4, if $\theta^\nu$ and $\theta^0$ are Gevrey-class $s$ solutions to \eqref{abstract active scalar eqn} for $\nu>0$ and $\nu=0$ respectively with initial datum $\theta_0$ on $\mathbb{T}^d\times[0,\bar{T}]$ with radius of convergence at least $\bar{\tau}$ as described in Theorem~\ref{Gevrey-class local wellposedness}, then there exists $T<\bar{T}$ and $\tau=\tau(t)<\bar{\tau}$ such that, for $t\in[0,T]$, we have
\begin{align}\label{convergence}
\lim_{\nu\rightarrow0}\|(\Lambda^re^{\tau\Lambda^\frac{1}{s}}\theta^{\nu}-\Lambda^re^{\tau\Lambda^\frac{1}{s}}\theta^0)(\cdot,t)\|_{L^2}=0.
\end{align}
\item Assume that the hypotheses and notations of Theorem~\ref{Local wellposedness in Sobolev space} are in force. Under the assumptions A3--A4, for $d\ge2$ and $s>\frac{d}{2}+1$ and $t\in[0,T]$, we have
\begin{align}\label{convergence sobolev}
\lim_{\nu\rightarrow0}\|(\theta^{\nu}-\theta^0)(\cdot,t)\|_{H^{s-1}}=0.
\end{align}
\end{itemize}
\end{theorem}
\section{Preliminaries}\label{preliminaries}

We introduce the following notations. We say $(\theta,u)$ is a weak solution to \eqref{abstract active scalar eqn} if they solve the system in the {\it weak sense}, that means for all $\phi\in C^{\infty}_0(\mathbb{T}^d\times(0,\infty),\R^d)$, we have
\begin{align*}
\int_0^\infty\int_{\mathbb{T}^d}(\partial_t \phi+u\cdot\nabla\phi)\theta(x,t)dxdt+\int_{\mathbb{T}^d}\phi(x,0)\theta_0(x)dx=\int_0^\infty\int_{\mathbb{T}^d}\phi S(x)dxdt.
\end{align*}
$W^{s,p}$ is the usual inhomogeneous Sobolev space with norm $\|\cdot\|_{W^{s,p} }$. For simplicity, we write $\|\cdot\|_{L^p}=\|\cdot\|_{L^p(\mathbb{T}^d)}$, $\|\cdot\|_{W^{s,p}}=\|\cdot\|_{W^{s,p}(\mathbb{T}^d)}$, etc. unless otherwise specified. We also write $H^s=W^{s,2}$.

\noindent We define $\|\cdot\|_{L.L. }$ to be the Log-Lipschitz norm given by
\begin{align*}
\|f\|_{L.L. }=\sup_{x\neq y}\frac{|f(x)-f(y)|}{|x-y|(1+|\log|x-y||)}.
\end{align*}

\noindent As in \cite{FV11b}, \cite{MV11}, for $s\ge1$, the Gevrey-class $s$ is given by
\begin{align*}
\bigcup_{\tau>0}\mathcal{D}(\Lambda^re^{\tau\Lambda^\frac{1}{s}}),
\end{align*}
for any $r\ge0$, where
\begin{align*}
\|\Lambda^re^{\tau\Lambda^\frac{1}{s}}f\|^2_{L^2}=\sum_{k\in\mathbb{Z}^d_*}|k|^{2r}e^{2\tau|k|^\frac{1}{s}}|\hat{f}(k)|^2,
\end{align*}
where $\tau=\tau(t)>0$ denotes the radius of convergence and $\Lambda=(-\Delta)^\frac{1}{2}$.

We also recall the following facts from the literature (see for example Azzam-Bedrossian \cite{AB11}, Bahouri-Chemin-Danchin \cite{BCD11} and Ziemer \cite{Z89}): there exists a constant $C>0$ such that

\begin{align}
\|f\|_{L.L. }&\le C \|\nabla f\|_{BMO},\label{L.L. bound in terms of BMO}\\
\|f\|_{BMO }&\le C \|f\|_{W^{1,d} },\label{BMO bound in terms of W}\\
\|f\|_{L^\infty }&\le C \|f\|_{W^{2,d} },\label{L infty bound 1}
\end{align}
and for $q>d$, there are constants $C(q)>0$ such that
\begin{align}
\|f\|_{L^\infty }&\le C(q) \|f\|_{W^{1,q} }.\label{L infty bound 2}
\end{align}
If $k>l$ and $k-\frac{d}{p}>l-\frac{d}{q}$, then we have
\begin{align}
\|f\|_{W^{l,q}}\le C\|f\|_{W^{k,p}}.\label{Kondrachov embedding theorem}
\end{align} 

\section{The non-diffusive active scalar equations}\label{The non-diffusive active scalar equations}

In this section, we study the non-diffusive equations \eqref{abstract active scalar eqn} for $\nu\ge0$. Depending on the values of $\nu$, we subdivide it into two cases, namely $\nu>0$ and $\nu=0$.

\subsection{The case where $\nu>0$}\label{Case for nu>0}

In this subsection we study the non-diffusive equations \eqref{abstract active scalar eqn} for $\nu>0$. First, given $\nu>0$ and initial datum $\theta_0\in W^{s,d}$ with $s>0$, we prove that  \eqref{abstract active scalar eqn} has a unique global-in-time solution $\theta^\nu\in W^{s,d}$. 

We begin with the following lemma which gives {\it a priori} bounds on $\theta^\nu$.
\begin{lemma}
Let $\nu>0$ and $\theta_0\in W^{s,p}$ for $s>0$ and $p>1$, and let $S=S(x)$ be a $C^\infty$-smooth source term. Then we have
\begin{align}\label{Sobolev estimate_1}
\|\theta^\nu(\cdot,t)\|_{L^p}\le \|\theta_0\|_{L^p}+t\|S\|_{L^\infty},
\end{align}
and we also have {\it a priori} bounds on $\theta^\nu$
\begin{align}\label{Sobolev estimate_2}
\|\theta^\nu(\cdot,t)\|_{W^{s,p}}\le C \|\theta_0\|_{W^{s,p}}\exp\left(C\int_0^t \|\nabla u^\nu(\cdot,\tilde t)\|_{L^\infty}d\tilde t+Ct\|S\|_{W^{s,p}}\right).
\end{align}
Here $C>0$ is a constant which depends on $p$ and the spatial dimension $d$.
\end{lemma}
\begin{proof}[\bf Proof]
The assertion \eqref{Sobolev estimate_1} follows from standard energy estimates. And for \eqref{Sobolev estimate_2}, we let $\Delta_k$'s be the dyadic blocks for $k\in\{-1\}\cup\N$. Applying $\Delta_k$ on \eqref{abstract active scalar eqn},
\begin{align*}
(\partial_t+u\cdot\nabla)(\Delta_k\theta^\nu)=R_k+\Delta_k S,
\end{align*}
where $R_k=u\cdot\nabla\Delta_k\theta^\nu-\Delta_k(u\cdot\nabla\theta^\nu)$. Since 
\begin{align*}
\|\Delta_k S(\cdot,t)\|_{L^p}+\|R_k(\cdot,t)\|_{L^p}\le CC_k(t)\left[2^{-ks}\|\nabla u^\nu(\cdot,t)\|_{L^\infty}\|\theta^\nu(\cdot,t)\|_{W^{s,p}}+\|S\|_{W^{s,p}}\right]
\end{align*}
with $\|C_k(\cdot,t)\|_{L^p}=1$, which implies \eqref{Sobolev estimate_2}. 
\end{proof}

\begin{proof}[\bf Proof of Theorem~\ref{Wellposedness in Sobolev space}]
We divide it into several cases.

\begin{case}$s=0$.
First, we observe that, by the assumption A3, $\partial_{x_i} T_{ij}^{\nu}$ is a smoothing operator of degree 2 for $\nu>0$. Hence with the help of the Fourier multiplier theorem (see Stein \cite{S70}), given $p>1$, there exists some constant $C=C(p,d)>0$ such that
\begin{align}\label{Fourier multiplier theorem}
\|u^\nu(\cdot,t)\|_{W^{2,p}}\le C \|\theta^\nu(\cdot,t)\|_{L^p}.
\end{align}
Together with \eqref{Sobolev estimate_1}, for $p>1$, if $\theta_0\in L^p$ and $S\in C^\infty$, then we have
\begin{align}\label{bound on u}
\|u^\nu(\cdot,t)\|_{W^{2,p}}\le C\left(\|\theta_0\|_{L^p}+t\|S\|_{L^\infty}\right),
\end{align}
where $t\|S\|_{L^\infty}=\int_{0}^{\infty}\|S(\cdot,\tilde{t})\|_{L^\infty}d\tilde{t}$. Next, using embedding theorems \eqref{L.L. bound in terms of BMO}-\eqref{L infty bound 1} and \eqref{bound on u} for $p=d$, we have
\begin{align}\label{infty bound}
\|u^\nu(\cdot,t)\|_{L^\infty}&\le C\|u^\nu(\cdot,t)\|_{W^{2,d}}\notag\\
&\le C\left(\|\theta_0\|_{L^p}+t\|S\|_{L^\infty}\right),
\end{align}
and
\begin{align}\label{LL bound}
\|u^\nu(\cdot,t)\|_{L.L.}&\le C \|\nabla u^\nu(\cdot,t)\|_{BMO}\notag\\
&\le C \|\nabla u^\nu(\cdot,t)\|_{W^{1,d}}\notag\\
&\le C \|u^\nu(\cdot,t)\|_{W^{2,d}}\notag\\
&\le C\left(\|\theta_0\|_{L^d}+t\|S\|_{L^\infty}\right),
\end{align} 
which shows that both $\|u^\nu(\cdot,t)\|_{L.L. }$ and $\|u^\nu(\cdot,t)\|_{L^\infty}$ are bounded in terms of $\theta_0$ and $S$. A bound on the Log-Lipschitzian norm of $u$ is essential to assure the existence and uniqueness of the {\it flow map}, and hence the existence and uniqueness of the solution. 

More precisely, to prove the existence of flow map, we consider the standard mollifier $\theta\in C^\infty_0 $, and we set $\theta_0^{(n)}=\theta_n *\theta_0$ for $n\in\N$ and $\theta_n(x)=n^d\theta(nx)$. By a standard argument, given $\nu>0$, we can obtain a sequence of global smooth solution $(\theta_{(n)}^\nu,u_{(n)}^\nu)$ to \eqref{abstract active scalar eqn} with $u_{(n)}^\nu=\partial_{x_i} T_{ij}^{\nu}[\theta_{(n)}^\nu]$. Define $\psi_n(x,t)$ to be the flow map given by
\begin{align}\label{flow map}
\partial_t\psi_n(x,t)=u_{(n)}^\nu(\psi_n(x,t),t).
\end{align}
One can show (for example in \cite{BK12}) that
\begin{align}\label{bound on flow map}
\|\psi_n(\cdot,t)\|_{*}\le C \exp\left(\int_0^t\|u_{(n)}^\nu(\cdot,\tilde{t})\|_{L.L.}d\tilde{t}\right),
\end{align}
where $C>0$ is independent of $\nu$ and $n$, and the norm $\|\cdot\|_{*}$ is given by
\begin{align*}
\|\psi\|_{*}=\sup_{x\neq y}\Phi(|\psi(x)-\psi(y)|,|x-y|)
\end{align*}
with
\begin{align*}
\Phi(r,s)=\left\{ \begin{array}{l}
\mbox{$\max\{\frac{1+|\log(s)|}{1+|\log(r)|},\frac{1+|\log(r)|}{1+|\log(s)|}\}$, if $(1-s)(1-r)\ge0$,}\\
\mbox{$(1+|\log(s)|)(1+|\log(r)|)$, if $(1-s)(1-r)\le0$.}
\end{array}\right.
\end{align*}
Using \eqref{LL bound} (with $u^\nu$ replaced by $u_{(n)}^\nu$) and \eqref{bound on flow map}, we obtain
\begin{align}\label{bound on psi 1}
|\psi_n(x_1,t)-\psi_n(x_2,t)|\le \alpha(t)|x_1-x_2|^{\beta(t)}
\end{align}
for all $(x_1,t),(x_2,t)\in\R^d\times\R^+$, where $\alpha(t),\beta(t)$ are some continuous functions which depends on $\theta_0$ and $S$. Furthermore, for $t_1,t_2\ge0$, using \eqref{infty bound} (with $u^\nu$ replaced by $u_{(n)}^\nu$),
\begin{align}\label{bound on psi 2}
|\psi_n(x,t_1)-\psi_n(x,t_2)|&\le C |t_2-t_1|(\|u_{(n)}^\nu(\cdot,t_1)\|_{L^\infty}+\|u_{(n)}^\nu(\cdot,t_2)\|_{L^\infty})\notag\\
&\le C |t_2-t_1|(\|\theta_0\|_{L^p}+\max\{t_1,t_2\}\|S\|_{L^\infty}).
\end{align}
Applying the estimates \eqref{bound on psi 1} and \eqref{bound on psi 2}, we see that the family $\{\psi_n\}_{n\in\N}$ is bounded and equicontinuous on every compact set in $\R^d\times\R^+$. By Arzela-Ascoli theorem, it implies the existence of a limiting trajectory $\psi(x,t)$ as $n\rightarrow\infty$. Performing the same analysis for $\{\psi_n^{-1}\}$, where $\psi_n^{-1}$ is the inverse of $\psi_n$, we see that $\psi(x,t)$ is a Lebesgue measure preserving homeomorphism with
\begin{align*}
\|\psi^{-1}(\cdot,t)\|_{*}=\|\psi(\cdot,t)\|_{*}&\le C \exp\left(\int_0^tC\left(\|\theta_0\|_{L^d}+\tilde{t}\|S\|_{L^\infty}\right)d\tilde{t}\right)\\
&\le C \exp\left(C\left(t\|\theta_0\|_{L^d}+t^2\|S\|_{L^\infty}\right)\right), \forall t\ge0.
\end{align*}
Define $\theta^\nu(x,t)=\theta_0(\psi^{-1}(x,t))$ and $u^\nu=\partial_{x_i} T_{ij}^{\nu}[\theta^\nu]$. The rest of the proof then follows from the one given in \cite{FS15}, which shows that $(\theta^\nu,u^\nu)$ is a weak solution to \eqref{abstract active scalar eqn}. 

To show that $(\theta^\nu,u^\nu)$ is unique, let $T>0$ and $\nu>0$, and suppose that $(\theta^{\nu,1},u^{\nu,1})$ and $(\theta^{\nu,2},u^{\nu,2})$ solve \eqref{abstract active scalar eqn} on $\mathbb{T}^d\times[0,T]$ with $\theta^{\nu,1}(\cdot,0)=\theta^{\nu,2}(\cdot,0)=\theta_0$. Following the similar argument given in \cite{BCD11}, there exists a constant $C>0$ such that for all $\delta\in(0,1)$ and $k\in\{-1\}\cup\N$, we have
\begin{align*}
&\|\Delta_k (\theta^{\nu,1}-\theta^{\nu,2})(\cdot,t)\|_{L^\infty}\\
&\le 2^{k\delta}(k+1)C(\|u^{\nu,1}(\cdot,t)\|_{\overline{L.L.}}+\|u^{\nu,2}(\cdot,t)\|_{\overline{L.L.}})\|(\theta^{\nu,1}-\theta^{\nu,2})(\cdot,t)\|_{B^{-\delta}_{\infty,\infty}}, \forall t\in[0,T],
\end{align*}
where $\|\cdot\|_{\overline{L.L.}}=\|\cdot\|_{L^\infty}+\|\cdot\|_{L.L.}$. We define
\begin{align*}
\bar{t}=\sup\left\{t\in[0,T]:C\int_0^t(\|u^{\nu,1}(\cdot,\tilde{t})\|_{\overline{L.L.}}+\|u^{\nu,2}(\cdot,\tilde{t})\|_{\overline{L.L.}})d\tilde{t}\le\frac{1}{2}\right\},
\end{align*}
then by the bounds \eqref{bound on u} and \eqref{infty bound}, $\bar{t}$ is well-defined. We let 
\begin{align*}
\delta_{\bar{t}}=C\int_0^{\bar{t}}(\|u^{\nu,1}(\cdot,\tilde{t})\|_{\overline{L.L.}}+\|u^{\nu,2}(\cdot,\tilde{t})\|_{\overline{L.L.}})d\tilde{t}.
\end{align*} 
Using Theorem~3.28 in \cite{BCD11}, for all $k\ge-1$ and $t\in[0,\bar{t}]$,
\begin{align*}
2^{-k\delta_{\bar{t}}}\|\Delta_k (\theta^{\nu,1}-\theta^{\nu,2})(\cdot,t)\|_{L^\infty}\le \frac{1}{2}\sup_{t\in[0,\bar{t}]}\|(\theta^{\nu,1}-\theta^{\nu,2})(\cdot,t)\|_{B^{-\delta_{\bar{t}}}_{\infty,\infty}}.
\end{align*}
Summing over $k$ and taking supremum over $[0,\bar{t}]$, we conclude that $\theta^{\nu,1}=\theta^{\nu,2}$ on $[0,\bar{t}]$. By repeating the argument a finite number of times, we obtain the uniqueness on the whole interval $[0,T]$. 
\end{case}

\begin{case} $s>0$. We only need {\it a priori} bounds on $\theta^\nu$. In view of \eqref{Sobolev estimate_2} with $p=d$, it remains to obtain a bound on $\|\nabla u^\nu(\cdot,t)\|_{L^\infty}$. We claim
\begin{align}\label{bound on nabla u}
\|\nabla u^\nu(\cdot,t)\|_{L^\infty}\le C (\|\theta_0\|_{W^{s,d}}+t\|S\|_{L^\infty}),\forall t\ge0.
\end{align}
We subdivide into two subcases.

\noindent{\bf Case 2(a):} $0<s<1$.
Define $q=\frac{d}{1-s}$. Then $q>d$ and using the embedding that $W^{s,d} \hookrightarrow L^q$, we have
\begin{align*}
\|\theta_0\|_{L^q}\le C \|\theta_0\|_{W^{s,d}}.
\end{align*}
Therefore using \eqref{L infty bound 2}, \eqref{Sobolev estimate_1} and \eqref{Fourier multiplier theorem} with $p=d$, we conclude that
\begin{align*}
\|\nabla u^\nu(\cdot,t)\|_{L^\infty}&\le C(q)\|u^\nu(\cdot,t)\|_{W^{2,q}}\\
&\le C (\|\theta_0\|_{W^{s,d}}+t\|S\|_{L^\infty}).
\end{align*}

\noindent{\bf Case 2(b):} $s\ge1$.
Using \eqref{Fourier multiplier theorem}, we take $p=2d$, which gives
\begin{align*}
\|\nabla u^\nu(\cdot,t)\|_{W^{1,2d}}\le C\|\theta^\nu(\cdot,t)\|_{L^{2d}}.
\end{align*}
On the other hand, we apply \eqref{Sobolev estimate_1}, and the embeddings $W^{1,2d} \hookrightarrow L^{\infty}$ and $W^{s,d} \hookrightarrow W^{\frac{1}{2},d} \hookrightarrow L^{2d}$ to get
\begin{align*}
\|\nabla u^\nu(\cdot,t)\|_{L^\infty}&\le C\|\nabla u^\nu(\cdot,t)\|_{W^{1,2d}}\\
&\le C \|\theta^\nu(\cdot,t)\|_{L^{2d}}\\
&\le C \left(\|\theta_0\|_{L^{2d}}+t\|S\|_{L^\infty}\right)\\
&\le C \left(\|\theta_0\|_{W^{\frac{1}{2},d}}+t\|S\|_{L^\infty}\right)\\
&\le C \left(\|\theta_0\|_{W^{s,d}}+t\|S\|_{L^\infty}\right).
\end{align*}
We substitute the above estimates on $\|\nabla u^\nu(\cdot,t)\|_{L^\infty}$ into \eqref{Sobolev estimate_2} with $p=d$ and obtain the desired {\it a priori} required bounds on $\theta^\nu$, namely
\begin{align}\label{Wsd solution}
\|\theta^\nu(\cdot,t)\|_{W^{s,d}}&\le C \|\theta_0\|_{W^{s,d}}\exp\left(C\int_0^t \|\nabla u^\nu(\cdot,\tilde t)\|_{L^\infty}d\tilde t+Ct\|S\|_{W^{s,d}}\right)\notag\\
&\le C \|\theta_0\|_{W^{s,d}}\exp\left(C\int_0^t \left(\|\theta_0\|_{W^{s,d}}+\tilde{t}\|S\|_{L^\infty}\right)d\tilde t+Ct\|S\|_{W^{s,d}}\right)\notag\\
&\le C \|\theta_0\|_{W^{s,d}}\exp\left(C\left(t\|\theta_0\|_{W^{s,d}}+t^2\|S\|_{L^\infty}+t\|S\|_{W^{s,d}}\right)\right).
\end{align}
Finally, the single exponential growth in time on $\|\nabla\theta(\cdot,t)\|_{L^d}$ follows readily from \eqref{Wsd solution} and Poinca\'{r}e inequality. We finish the proof of Theorem~\ref{Wellposedness in Sobolev space}. 
\end{case}
\end{proof}
\begin{remark}\label{existence of Hr solution}
For $r>0$, if we take $s>\max\{\frac{2r+2-d}{2},r\}$, then by the Sobolev embedding theorem \eqref{Kondrachov embedding theorem} and the bound \eqref{Wsd solution}, we have
\begin{align*}
\|\theta^\nu(\cdot,t)\|_{H^{r}}&\le C\|\theta^\nu(\cdot,t)\|_{W^{s,d}}\notag\\
&\le C \|\theta_0\|_{W^{s,d}}\exp\left(C\left(t\|\theta_0\|_{W^{s,d}}+t^2\|S\|_{L^\infty}+t\|S\|_{W^{s,d}}\right)\right),
\end{align*}
where $H^r=W^{r,2}$. Hence Theorem~\ref{Wellposedness in Sobolev space} implies that there exists a unique global-in-time $H^r$ solution $\theta^\nu$ to \eqref{abstract active scalar eqn} whenever $\theta_0\in W^{s,d}$ and $S\in C^\infty$ for $r>0$ and $s>\max\{\frac{2r+2-d}{2},r\}$.
\end{remark}
Next we study the Gevrey-class $s$ solutions to \eqref{abstract active scalar eqn} for $\nu>0$ when the initial datum $\theta_0$ and forcing term $S$ are in the same Gevrey-class. Recall that for $\nu>0$, there exists a constant $C_{\nu}>0$ such that for all $1\le i,j\le d$, $$|\widehat T^{\nu}_{ij}(k)|\le C_{\nu}|k|^{-3}, \forall k\in\mathbb{Z}^d,$$
which gives 2-orders of smoothing 
\begin{equation}\label{2 order smoothing}
\|u^\nu\|_{H^2}\le C_{\nu}\|\theta^\nu\|_{L^2}.
\end{equation}
To prove the global-in-time existence as claimed in Theorem~\ref{Gevrey-class global wellposedness}, we give the estimates of $\theta^\nu$ as follows. We take $L^2$-inner product of \eqref{abstract active scalar eqn}$_1$ with $e^{2\tau\Lambda^{\frac{1}{s}}}\theta^\nu$ and obtain
\begin{align}\label{a priori estimates on theta}
\frac{1}{2}\frac{d}{dt}\|e^{\tau\Lambda^{\frac{1}{s}}}\theta^\nu\|^2_{L^2}-\dot{\tau}\|\Lambda^\frac{1}{2s}e^{\tau\Lambda^{\frac{1}{s}}}\theta^\nu\|^2_{L^2}&=-\langle u^\nu\cdot\nabla\theta^\nu,e^{2\tau\Lambda^{\frac{1}{s}}}\theta^\nu\rangle+\langle S,e^{2\tau\Lambda^{\frac{1}{s}}}\theta^\nu\rangle\notag\\
&\le\Big|-\langle u^\nu\cdot\nabla\theta^\nu,e^{2\tau\Lambda^{\frac{1}{s}}}\theta^\nu\rangle\Big|+\|e^{\tau\Lambda^\frac{1}{s}}S\|_{L^2}\|e^{\tau\Lambda^\frac{1}{s}}\theta^\nu\|_{L^2}.
\end{align}
The following lemma gives the estimates on the term $-\langle u^\nu\cdot\nabla\theta^\nu,e^{2\tau\Lambda^{\frac{1}{s}}}\theta^\nu\rangle$.
\begin{lem}\label{estimate on stretching term}
For $\nu>0$ and $s\ge1$, we have
\begin{equation}\label{estimate on cross term}
\left|-\langle e^{\tau\Lambda^{\frac{1}{s}}}(u^\nu\cdot\nabla\theta^\nu),e^{\tau\Lambda^{\frac{1}{s}}}\theta^\nu\rangle\right|\le C\tau\|e^{\tau\Lambda^\frac{1}{s}}\theta^\nu\|_{L^2}\|\Lambda^\frac{1}{2s}e^{\tau\Lambda^\frac{1}{s}}\theta^\nu\|^2_{L^2},
\end{equation}
where $C=C(\nu)>0$ depends on $\nu$.
\end{lem}
\begin{proof}[\bf Proof]
The proof is reminiscent of the one given in \cite{MV11}. Since $\nabla\cdot u^\nu=0$ we have
\begin{align*}
\langle u^\nu\cdot\nabla e^{\tau\Lambda^{\frac{1}{s}}}\theta^\nu,e^{\tau\Lambda^{\frac{1}{s}}}\theta^\nu\rangle=0,
\end{align*} 
which gives
\begin{align*}
\Big|\langle u^\nu\cdot\nabla\theta^\nu,e^{2\tau\Lambda^{\frac{1}{s}}}\theta^\nu\rangle\Big|&=\Big|\langle u^\nu\cdot\nabla\theta^\nu,e^{2\tau\Lambda^{\frac{1}{s}}}\theta^\nu\rangle-\langle u^\nu\cdot\nabla e^{\tau\Lambda^{\frac{1}{s}}}\theta^\nu,e^{\tau\Lambda^{\frac{1}{s}}}\theta^\nu\rangle\Big|\\
&=\Big|i(2\pi)^d\sum_{j+k=l}(\hat{u}_j\cdot k)(\hat{\theta}_k\cdot\bar{\hat{\theta}}_l)e^{\tau|l|^\frac{1}{s}}(e^{\tau|l|^\frac{1}{s}}-e^{\tau|k|^\frac{1}{s}})\Big|.
\end{align*}
We make use of the inequality $e^x-1\le xe^x$ for $x\ge0$ and the triangle inequality $|k+j|^\frac{1}{s}\le|k|^\frac{1}{s}+|j|^\frac{1}{s}$ to obtain
\begin{align*}
\Big|e^{\tau|l|^\frac{1}{s}}-e^{\tau|k|^\frac{1}{s}}\Big|\le C\tau\frac{|j|}{|k|^{1-\frac{1}{s}}+|l|^{1-\frac{1}{s}}}e^{\tau|l|^\frac{1}{s}}e^{\tau|k|^\frac{1}{s}}.
\end{align*}
Hence we have
\begin{align*}
\Big|\langle u^\nu\cdot\nabla\theta^\nu,e^{2\tau\Lambda^{\frac{1}{s}}}\theta^\nu\rangle\Big|&\le C\tau\sum_{j+k=l}|k||\hat{u}_j|e^{\tau|j|^\frac{1}{s}}|\hat{\theta}_k|e^{\tau|k|^\frac{1}{s}}|\hat{\theta}_l|e^{\tau|l|^\frac{1}{s}}\frac{|j|}{|k|^{1-\frac{1}{s}}+|l|^{1-\frac{1}{s}}}\\
&\le C\tau\sum_{j+k=l}|j||\hat{u}_j|e^{\tau|j|^\frac{1}{s}}|\hat{\theta}_k|e^{\tau|k|^\frac{1}{s}}|\hat{\theta}_l|e^{\tau|l|^\frac{1}{s}}|k|^\frac{1}{2s}(|k|^\frac{1}{2s}+|l|^\frac{1}{2s})\\
&\le C\tau\|e^{\tau\Lambda^\frac{1}{s}}\theta^\nu\|_{L^2}\|\Lambda^\frac{1}{2s}e^{\tau\Lambda^\frac{1}{s}}\theta^\nu\|_{L^2}\sum_{j\neq0}|j|^{1+\frac{1}{2s}}|\widehat{u}(j)|e^{\tau|j|^\frac{1}{s}}\notag\\
&\qquad+C\tau\|\Lambda^\frac{1}{2s}e^{\tau\Lambda^\frac{1}{s}}\theta^\nu\|^2_{L^2}\sum_{j\neq0}|j||\widehat{u^\nu}(j)|e^{\tau|j|^\frac{1}{s}}\notag\\
&\le C\tau\|e^{\tau\Lambda^\frac{1}{s}}\theta^\nu\|_{L^2}\|\Lambda^\frac{1}{2s}e^{\tau\Lambda^\frac{1}{s}}\theta^\nu\|_{L^2}\|\Lambda^{2+\frac{1}{2s}}e^{\tau\Lambda^\frac{1}{s}}u^\nu\|_{L^2}\notag\\
&\qquad+C\tau\|\Lambda^\frac{1}{2s}e^{\tau\Lambda^\frac{1}{s}}\theta^\nu\|^2_{L^2}\|\Lambda^2e^{\tau\Lambda^\frac{1}{s}}u^\nu\|_{L^2},
\end{align*}
where we have used the fact that $\sum_{j\neq0,j\in\mathbb{Z}}|j|^{-2}<\infty$. Using the property \eqref{2 order smoothing}, we have
\begin{equation*}
\mbox{$\|\Lambda^{2+\frac{1}{2s}}e^{\tau\Lambda^\frac{1}{s}}u^\nu\|_{L^2}\le C_{\nu}\|\Lambda^\frac{1}{2s}e^{\tau\Lambda^\frac{1}{s}}\theta^\nu\|_{L^2}$ and $\|\Lambda^2 e^{\tau\Lambda^\frac{1}{s}}u^\nu\|_{L^2}\le C_{\nu}\|e^{\tau\Lambda^\frac{1}{s}}\theta^\nu\|_{L^2}$.}
\end{equation*}
Hence there is $C=C(\nu)>0$ such that
\begin{equation*}
\left|-\langle e^{\tau\Lambda^{\frac{1}{s}}}(u^\nu\cdot\nabla\theta^\nu),e^{\tau\Lambda^{\frac{1}{s}}}\theta^\nu\rangle\right|\le C\tau\|e^{\tau\Lambda^\frac{1}{s}}\theta^\nu\|_{L^2}\|\Lambda^\frac{1}{2s}e^{\tau\Lambda^\frac{1}{s}}\theta^\nu\|^2_{L^2},
\end{equation*}
which finishes the proof of \eqref{estimate on cross term}.
\end{proof}
To complete the proof of Theorem~\ref{Gevrey-class global wellposedness}, we apply \eqref{estimate on cross term} on \eqref{a priori estimates on theta} to obtain
\begin{align*}
&\frac{1}{2}\frac{d}{dt}\|e^{\tau\Lambda^{\frac{1}{s}}}\theta^\nu\|^2_{L^2}-\dot{\tau}\|\Lambda^{\frac{1}{2s}}e^{\tau\Lambda^{\frac{1}{s}}}\theta^\nu\|^2_{L^2}\notag\\
&\le C\tau\|e^{\tau\Lambda^\frac{1}{s}}\theta^\nu\|_{L^2}\|\Lambda^\frac{1}{2s}e^{\tau\Lambda^\frac{1}{s}}\theta^\nu\|^2_{L^2}+\|e^{\tau\Lambda^\frac{1}{s}}S\|_{L^2}\|e^{\tau\Lambda^\frac{1}{s}}\theta^\nu\|_{L^2}.
\end{align*}
Choose $\tau>0$ such that 
\begin{equation*}
\dot{\tau}+C\tau\|e^{\tau\Lambda^\frac{1}{s}}\theta^\nu\|_{L^2}=0,
\end{equation*}
then we have
\begin{equation*}
\frac{1}{2}\frac{d}{dt}\|e^{\tau\Lambda^{\frac{1}{s}}}\theta^\nu\|^2_{L^2}\le\|e^{\tau\Lambda^\frac{1}{s}}S\|_{L^2}\|e^{\tau\Lambda^{\frac{1}{s}}}\theta^\nu\|_{L^2},
\end{equation*}
which gives 
\begin{equation*}
\|e^{\tau(t)\Lambda^\frac{1}{s}}\theta^\nu(t)\|_{L^2}\le\|e^{\tau_0\Lambda^\frac{1}{s}}\theta_0\|_{L^2}+2\|e^{\tau_0\Lambda^\frac{1}{s}}S\|_{L^2}.
\end{equation*}
Hence $\tau(t)$ satisfies
\begin{equation*}
\tau(t)\ge\tau_0e^{-C\left(\|e^{\tau_0\Lambda^\frac{1}{s}}\theta_0\|_{L^2}+2\|e^{\tau_0\Lambda^\frac{1}{s}}S\|_{L^2}\right)t}
\end{equation*}
and the proof of Theorem~\ref{Gevrey-class global wellposedness} is complete.
\begin{remark}
We notice that for the ``diffusive'' case, i.e. for the following system when $\kappa>0$:
\begin{align}
\label{abstract active scalar eqn diffusive} \left\{ \begin{array}{l}
\partial_t\theta^{\nu}+u^\nu\cdot\nabla\theta^{\nu}=\kappa\Delta\theta^{\nu}, \\
u_j^{\nu}=\partial_{x_i} T_{ij}^{\nu}[\theta^{\nu}],\theta^{\nu}(x,0)=\theta_0(x),
\end{array}\right.
\end{align}
one can obtain global-in-time existence of solution Gevrey class $s\ge1$ with lower bound
on $\tau(t)$ that does not vanish as $t\rightarrow\infty$. To see it, we apply the previous analysis on \eqref{abstract active scalar eqn diffusive} to obtain
\begin{align}
\frac{1}{2}\frac{d}{dt}\|e^{\tau\Lambda^{\frac{1}{s}}}\theta^\nu\|^2_{L^2}-\dot{\tau}\|\Lambda^{\frac{1}{2s}}e^{\tau\Lambda^{\frac{1}{s}}}\theta^\nu\|^2_{L^2}+\kappa\|e^{\tau\Lambda^\frac{1}{s}}\theta^{\nu}\|^2_{L^2}\le C\tau\|e^{\tau\Lambda^\frac{1}{s}}\theta^\nu\|_{L^2}\|\Lambda^\frac{1}{2s}e^{\tau\Lambda^\frac{1}{s}}\theta^\nu\|^2_{L^2}.
\end{align}
Choosing $\tau>0$ such that
\begin{equation*}
\dot{\tau}+C\tau\|e^{\tau\Lambda^\frac{1}{s}}\theta^\nu\|_{L^2}=0,
\end{equation*}
then we have
\begin{equation*}
\frac{1}{2}\frac{d}{dt}\|e^{\tau\Lambda^{\frac{1}{s}}}\theta^\nu\|^2_{L^2}+\kappa\|e^{\tau\Lambda^\frac{1}{s}}\theta^{\nu}\|^2_{L^2}\le0.
\end{equation*}
Hence we obtain
\begin{equation*}
\|e^{\tau(t)\Lambda^\frac{1}{s}}\theta^\nu(t)\|_{L^2}\le\|e^{\tau_0\Lambda^\frac{1}{s}}\theta_0\|_{L^2}e^{\frac{-\kappa t}{2}},
\end{equation*}
and 
\begin{equation*}
\tau(t)\ge\tau(0)e^{-C\|e^{\tau_0\Lambda^\frac{1}{s}\theta_0}\|_{L^2}\int_0^te^{\frac{-\kappa s}{2}ds}}\ge e^{-\frac{2C}{\kappa}\|e^{\tau_0\Lambda^\frac{1}{s}\theta_0}\|_{L^2}}.
\end{equation*}
Observe that the lower bound $e^{-\frac{2C}{\kappa}\|e^{\tau_0\Lambda^\frac{1}{s}\theta_0}\|_{L^2}}$ tends to zero as $\kappa\rightarrow0$.
\end{remark}

\subsection{The case where $\nu=0$}\label{Case for nu=0}

In this subsection we study the non-diffusive equations \eqref{abstract active scalar eqn} for $\nu=0$. Recall that we consider the following active scalar equation
\begin{align}
\label{abstract active scalar eqn nu=0} \left\{ \begin{array}{l}
\partial_t\theta^0+u^0\cdot\nabla\theta^0=S, \\
u^0_j=\partial_{x_i} T^0_{ij}[\theta^0],\theta^0(x,0)=\theta_0(x)
\end{array}\right.
\end{align}
where $T_{ij}^{0}$ is an operator which satisfies assumptions A1--A2 and either A$5_1$ or A$5_2$. Based on the assumptions A$5_1$ and A$5_2$, we consider the following two cases separately:
\subsubsection{When A$\bold 5_1$ is in force.}
Different from the case for $\nu>0$, as it was proved in \cite{FV11b}, the equation \eqref{abstract active scalar eqn nu=0} is {\it ill-posed} in the sense of Hadamard, which means that the solution map associated to the Cauchy problem for \eqref{abstract active scalar eqn nu=0} is not Lipschitz continuous with respect to perturbations in the initial datum around a specific steady profile $\theta_0$, in the topology of a certain Sobolev space $X$. Nevertheless, as pointed out in \cite{FV11b}, it is possible to obtain the local existence and uniqueness of solutions to \eqref{abstract active scalar eqn nu=0} in spaces of real-analytic functions, owing to the fact that the derivative loss in the nonlinearity $u^0\cdot\nabla\theta^0$ is of order at most one (both in $u^0$ and in $\nabla\theta^0$).

In the present work, we extend the results of \cite{FV11b} to the case of Gevrey-class solutions. We first state and prove the following proposition which gives the Gevrey-class local wellposedness for \eqref{abstract active scalar eqn nu=0}.
\begin{proposition}\label{Gevrey-class local existence nu=0}
Fix $s\ge1$ and $K_0>0$. Let $\theta_0$ and $S$ be of Gevrey-class $s$ with radius of convergence $\tau_0>0$ and
\begin{align}\label{bounds on gevrey norm of S and theta0}
\|\Lambda^re^{\tau_0\Lambda^\frac{1}{s}}\theta^0(\cdot,0)\|_{L^2}\le K_0,\qquad\|\Lambda^re^{\tau_0\Lambda^\frac{1}{s}}S\|_{L^2}\le K_0,
\end{align}
where $r>\frac{d}{2}+\frac{3}{2}$. There exists $T_*=T_*(\tau_0, K_0) > 0$ and a unique Gevrey-class $s$ solution on $[0, T_*)$ to the initial value problem associated to \eqref{abstract active scalar eqn nu=0}.
\end{proposition}
\begin{proof}[\bf Proof]
The idea of the proof follows by a similar argument given in \cite{FV11b}. For $r>\frac{d}{2}+\frac{3}{2}$, we define
\begin{equation*}
\|\theta^0\|_{\tau,r}^2=\|\Lambda^re^{\tau\Lambda^\frac{1}{s}}\theta^0\|_{L^2}^2=\sum_{k\in\mathbb{Z}^d_*}|k|^{2r}e^{2\tau|k|^\frac{1}{s}}|\widehat{\theta^0}(k)|^2
\end{equation*}
We take $L^2$-inner product of \eqref{abstract active scalar eqn nu=0} with $\Lambda^{2r}e^{2\tau\Lambda^{\frac{1}{s}}}\theta^0$ and obtain
\begin{align}\label{a priori estimates on theta singular gevrey}
\frac{1}{2}\frac{d}{dt}\|\theta^0\|^2_{\tau,r}-\dot{\tau}\|\Lambda^{\frac{1}{2s}}\theta^0\|^2_{\tau,r}=\langle u^0\cdot\nabla\theta^0,\Lambda^{2r}e^{2\tau\Lambda^{\frac{1}{s}}}\theta^0\rangle-\langle \Lambda^re^{\tau\Lambda^{\frac{1}{s}}}S,\Lambda^re^{\tau\Lambda^{\frac{1}{s}}}\theta^0\rangle.
\end{align}
Write $\mathcal{R}=\langle u\cdot\nabla\theta^0,\Lambda^{2r}e^{2\tau\Lambda^{\frac{1}{s}}}\theta^0\rangle$, then it can rewritten as
\begin{align*}
\mathcal{R}=i(2\pi)^d\sum_{j+k=l,j,k,l\in\mathbb{Z}^d_*}\widehat{u^0}(j)\cdot k\widehat{\theta^0}(k)|l|^{2r}e^{2\tau|l|^\frac{1}{s}}\widehat{\theta^0}(-l).
\end{align*}
Using the assumption A2 that $|\widehat{u^0}(j)|\le C|j||\widehat{\theta^0}(j)|$ and the fact $|l|^\frac{1}{s}=|j+k|^\frac{1}{s}\le |j|^\frac{1}{s}+|k|^\frac{1}{s}$ for $s\ge1$ and $|j|,|k|\ge1$, we have
\begin{align*}
\mathcal{R}&\le C\sum|j||k|(|j|^r+|k|^r)|\widehat{\theta^0}(j)|e^{\tau|j|^\frac{1}{s}}|\widehat{\theta^0}(k)|e^{\tau|k|^\frac{1}{s}}|l|^r|\widehat{\theta^0}(l)|e^{\tau|l|^\frac{1}{s}}\\
&\le C\sum(|j|^{r+\frac{1}{2}}|k|^\frac{3}{2}+|k|^{r+\frac{1}{2}}|j|^\frac{3}{2})|\widehat{\theta^0}(j)|e^{\tau|j|^\frac{1}{s}}|\widehat{\theta^0}(k)|e^{\tau|k|^\frac{1}{s}}|l|^{r+\frac{1}{2}}e^{\tau|l|^\frac{1}{s}}\\
&\le C\|\Lambda^\frac{1}{2s}\theta^0\|^2_{\tau,r}\sum|j|^\frac{3}{2}|\widehat{\theta^0}(j)|e^{\tau|j|^\frac{1}{s}}\\
&\le C\|\Lambda^\frac{1}{2s}\theta^0\|^2_{\tau,r}\|\theta^0\|_{\tau,r},
\end{align*}
where the last inequality follows since $r>\frac{d}{2}+\frac{3}{2}$. Hence we obtain from \eqref{a priori estimates on theta singular gevrey} that
\begin{align}\label{a priori bound on theta singular gevrey}
\frac{1}{2}\frac{d}{dt}\|\theta^0(\cdot,t)\|^2_{\tau,r}\le (\dot{\tau}(t)+C\|\theta^0(\cdot,t)\|_{\tau,r})\|\Lambda^\frac{1}{2s}\theta^0\|^2_{\tau,r}+\|S\|_{\tau,r}\|\theta^0(\cdot,t)\|_{\tau,r}.
\end{align}
Let $\tau(t)$ be deceasing and satisfy 
\begin{align*}
\dot{\tau}+4CK_0=0,
\end{align*}
with initial condition $\tau(0)=\tau_0$, then we have $\dot{\tau}(t)+C\|\theta^0(\cdot,t)\|_{\tau,r}<0$, and from \eqref{a priori bound on theta singular gevrey} that 
\begin{align}\label{bound on solution nu=0}
\|\theta^0(\cdot,t)\|_{\tau,r}\le\|\theta^0(\cdot,0)\|_{\tau,r}+2\|S\|_{\tau,r}=3K_0
\end{align}
as long as $\tau(t)>0$. Hence it implies the existence of a Gevrey-class $s$ solution $\theta^0$ on $[0,T_*)$, where the maximal time of existence of the Gevrey-class $s$ solution is given by $T_*=\frac{\tau_0}{4CK_0}$.
\end{proof}
\begin{proof}[\bf Proof of Theorem~\ref{Gevrey-class local wellposedness}]
By choosing $\bar{T}=\frac{T_*}{2}$ and $\bar{\tau}=\tau(\bar{T})$, where $T_*$, $\tau(t)$ are as defined in Proposition~\ref{Gevrey-class local existence nu=0}, if $\theta(\cdot,0)=\theta_0$ and $S$ be of Gevrey-class $s$, both with radius of convergence at least $\tau_0$ and satisfy \eqref{bounds on gevrey norm of S and theta0}, then there exists a unique Gevrey-class $s$ solution $\theta^0$ to \eqref{abstract active scalar eqn} for $\nu=0$ defined on $[0,\bar{T}]$ with radius of convergence at least $\bar{\tau}$. The time $\bar{T}$ and radius on convergence $\bar{\tau}$ should only depend on $C_0$ as described in assumption A3, hence they can be chosen independent of $\nu$ and the proof of Proposition~\ref{Gevrey-class local existence nu=0} also applies to \eqref{abstract active scalar eqn} for $\nu>0$. The bounds \eqref{bound on solution 2}-\eqref{bound on solution 1} follow immediately from \eqref{bound on solution nu=0}.
\end{proof}
\begin{remark}
By uniqueness, for $\nu>0$, the Gevrey-class $s$ solution $\theta^\nu$ as obtained in Theorem~\ref{Gevrey-class global wellposedness} coincides with the one as obtained in Theorem~\ref{Gevrey-class local wellposedness} on $\mathbb{T}^d\times[0,\bar{T}]$. 
\end{remark}
\subsubsection{When A$\bold 5_2$ is in force.}
Contrary to the previous case, when assumption A$5_2$ is in force, the operator $\partial_x T^0$ becomes a zero order operator with $\partial_x T^0:L^2\rightarrow L^2$ being bounded. Following the idea given in \cite{FRV12}, we show that under the assumptions A1--A2 and A$5_2$, the equation \eqref{abstract active scalar eqn nu=0} is locally wellposed in Sobolev space $H^s$ for $s>\frac{d}{2}+1$, thereby proving Theorem~\ref{Local wellposedness in Sobolev space}.

Before we give the proof of Theorem~\ref{Local wellposedness in Sobolev space}, we recall the following proposition from \cite{FRV12}:
\begin{proposition}
Suppose that $s>0$ and $p\in(1,\infty)$. If $f,g\in\mathcal{S}$, then
\begin{equation}\label{Commutator estimate}
\|\Lambda^s(fg)-f\Lambda^s g\|_{L^p}\le C\Big(\|\nabla f\|_{L^{p_1}}\|\Lambda^{s-1} g\|_{L^{p_2}}+\|\Lambda^s f\|_{L^{p_3}}\|g\|_{L^{p+4}}\Big),
\end{equation}
where $\Lambda=(-\Delta)^\frac{1}{2}$, $\frac{1}{p}=\frac{1}{p_1}+\frac{1}{p_2}=\frac{1}{p_3}+\frac{1}{p_4}$, and $p,p_2,p_3\in(1,\infty)$.
\end{proposition}
\begin{proof}[\bf Proof of Theorem~\ref{Local wellposedness in Sobolev space}]
For simplicity, we denote $\theta^0$ and $u^0$ by $\theta$ and $u$ respectively. We subdivide the proof into 3 steps.
\begin{step}
We consider the sequence of approximations $\{\theta_n\}_{n\ge1}$ given by the solutions of
\begin{align}\label{theta n=1}
\partial_t\theta_1&=S\notag\\
\theta_1(\cdot,0)&=\theta_0.
\end{align}
and
\begin{align}\label{theta n=n}
\partial_t\theta_n+u_{n-1}\cdot\nabla\theta_n&=S\notag\\
\theta_{n-1}&=\partial_xT^0[\theta_{n-1}]\\
\theta_n(\cdot,0)&=\theta_0.\notag
\end{align}
System \eqref{theta n=1} can be solved easily and for all $T>0$, we also have the bound
\begin{align*}
\|\Lambda^s\theta_1\|^2_{L^\infty(0,T;L^2)}\le\|\Lambda^s\theta_0\|^2_{L^2}+T\|S\|^2_{H^s},
\end{align*}
where $\Lambda=(-\Delta)^\frac{1}{2}$. For the system \eqref{theta n=n}, we consider the linear approximated system 
\begin{align*}
\partial_t\theta^\varepsilon+v\cdot\nabla\theta^\varepsilon-\varepsilon\Delta\theta^\varepsilon&=S\notag\\
\theta^\varepsilon(\cdot,0)&=\theta_0,
\end{align*}
and the details follow from Theorem~A1 in \cite{FRV12}. This shows that there exists a unique solution $\theta_n\in L^\infty(0,T;H^s)$ of \eqref{theta n=n}.
\end{step}
\begin{step}
Next we show that $\{\theta_n\}_{n\ge0}$ is bounded. We fix a time $T$ (to be chosen later) such that $$T<\frac{\|\Lambda^s\theta_0\|_{L^2}^2}{\|S\|^2_{H^s}}.$$
Assume that
\begin{align}\label{bound on theta n}
\|\Lambda^s\theta_j\|^2_{L^\infty(0,T;L^2)}\le 2\|\Lambda^s\theta_0\|^2_{L^2},
\end{align}
for $1\le j\le n-1$. By A$5_2$, we have
\begin{align*}
\|\Lambda^s u_{n-1}(\cdot,t)\|^2_{L^2}\le\|\Lambda^s \theta_{n-1}(\cdot,t)\|^2_{L^2}.
\end{align*}
for all $t>0$. Apply $\Lambda^s$ on \eqref{theta n=n} and take inner product with $\Lambda^s\theta_n$, we obain
\begin{align}\label{estimate on theta n identity}
\frac{1}{2}\frac{d}{dt}\intox|\Lambda^s\theta_n|^2+\intox\Lambda^s\theta_n\cdot\Lambda^s(u_{n-1}\cdot\nabla\theta_n)=\intox\Lambda^s\theta_n\cdot\Lambda^s S.
\end{align}
The term $\dis\intox\Lambda^s\theta_n\cdot\Lambda^s(u_{n-1}\cdot\nabla\theta_n)$ can be rewritten as
\begin{align*}
&\intox\Lambda^s\theta_n\cdot\Lambda^s(u_{n-1}\cdot\nabla\theta_n)\\
&=\intox\Lambda^s\theta_n\cdot(\Lambda^s(u_{n-1}\cdot\nabla\theta_n)-u_{n-1}\cdot\Lambda^s(\nabla\theta_n))+\intox\Lambda^s\theta_n\cdot u_{n-1}\Lambda^s(\nabla\theta_n),
\end{align*}
Upon integration by parts, the term $\dis \intox\Lambda^s\theta_n\cdot u_{n-1}\Lambda^s(\nabla\theta_n)$ vanishes since $\nabla\cdot u_{n-1}=0$. Using \eqref{Commutator estimate} for $f=u_{n-1}$, $g=\nabla\theta_n$, $p=2$, $p_1=\infty$, $p_2=2$, $p_3=2$, $p_4=\infty$, and applying the assumption A$5_2$, we have
\begin{align}\label{commutator estimate}
&\left|\intox\Lambda^s\theta_n\cdot(\Lambda^s(u_{n-1}\cdot\nabla\theta_n)-u_{n-1}\cdot\Lambda^s(\nabla\theta_n))\right|\notag\\
&\le\|\Lambda^s\theta_n\|_{L^2}\|\Lambda^s(u_{n-1}\cdot\nabla\theta_n)-u_{n-1}\cdot\Lambda^s(\nabla\theta_n)\|_{L^2}\notag\\
&\le C\|\Lambda^s\theta_n\|_{L^2}(\|\nabla u_{n-1}\|_{L^\infty}\|\Lambda^s\theta_n\|_{L^2}+\|\Lambda^s u_{n-1}\|_{L^2}\|\nabla\theta_n\|_{L^\infty})\notag\\
&\le C\|\Lambda^s\theta_n\|_{L^2}(\|\Lambda^s \theta_{n-1}\|_{L^2}\|\Lambda^s\theta_n\|_{L^2}+\|\Lambda^s \theta_{n-1}\|_{L^2}\|\Lambda^s\theta_n\|_{L^2})\notag\\
&\le 2C\|\Lambda^s\theta_n\|_{L^2}(\|\Lambda^s \theta_{n-1}\|_{L^2}\|\Lambda^s\theta_n\|_{L^2}).
\end{align}
Hence we obtain from \eqref{estimate on theta n identity} and \eqref{commutator estimate} that
\begin{align}\label{estimate on theta n 1}
\frac{1}{2}\frac{d}{dt}\|\Lambda^s \theta_{n}(\cdot,t)\|^2_{L^2}\le 2C\|\Lambda^s\theta_n\|_{L^2}(\|\Lambda^s \theta_{n-1}\|_{L^2}\|\Lambda^s\theta_n\|_{L^2})+\|\Lambda^s\theta_n\|_{L^2}\|\Lambda^sS\|_{L^2}.
\end{align}
Applying the bound \eqref{bound on theta n} on $\theta_n$, we have
\begin{align}\label{estimate on theta n 2}
\frac{d}{dt}\|\Lambda^s \theta_{n}(\cdot,t)\|_{L^2}\le4C\|\Lambda^s \theta_0\|_{L^2}\|\Lambda^s\theta_n\|_{L^2}+2\|\Lambda^sS\|_{L^2},
\end{align}
and hence by integrating \eqref{estimate on theta n 2} over $t$ and choosing $T$ small enough, \eqref{bound on theta n} also holds for $j=n$.
\end{step}
\begin{step}
Finally, we show that $\{\theta_n\}_{n\ge0}$ is a Cauchy sequence. Denote the difference of $\theta_n$ and $\theta_{n-1}$ by 
$$\tilde\theta_{n}=\theta_n-\theta_{n-1}.$$
It follows from \eqref{theta n=n} that $\tilde\theta_{n}$ satisfies
\begin{align}\label{tilde theta n}
\partial_t\tilde\theta_n+u_{n-1}\cdot\nabla\tilde\theta_n+\tilde u_{n-1}\cdot\nabla\theta_{n-1}=0,
\end{align}
where $\tilde u_{n-1}=\partial_xT^0[\tilde\theta_{n-1}]$. Apply $\Lambda^{s-1}$ on \eqref{tilde theta n} and take inner product with $\Lambda^{s-1}\tilde\theta_n$,
\begin{align}\label{estimate on tilde theta n identity}
\frac{1}{2}\frac{d}{dt}\intox|\Lambda^{s-1}\tilde\theta_n|^2+\intox\Lambda^{s-1}\tilde\theta_n\cdot\Lambda^{s-1}(u_{n-1}\cdot\nabla\tilde\theta_n)+\intox\Lambda^{s-1}\tilde\theta_n\cdot\Lambda^{s-1}(\tilde u_{n-1}\cdot\nabla\theta_{n-1})=0.
\end{align}
The term $\dis\intox\Lambda^{s-1}\tilde\theta_n\cdot\Lambda^{s-1}(u_{n-1}\cdot\nabla\tilde\theta_n)$ can be estimated as follows.
\begin{align*}
&\left|\intox\Lambda^{s-1}\tilde\theta_n\cdot\Lambda^{s-1}(u_{n-1}\cdot\nabla\tilde\theta_n)\right|\\
&\le\left|\intox\Lambda^{s-1}\tilde\theta_n\cdot(\Lambda^{s-1}(u_{n-1}\cdot\nabla\tilde\theta_n)-u_{n-1}\cdot\Lambda^{s-1}(\nabla\tilde\theta_n)\right|+\left|\intox \Lambda^{s-1}\tilde\theta_n\cdot u_{n-1}\cdot\Lambda^{s-1}(\nabla\tilde\theta_n)\right|\\
&\le C\|\Lambda^{s-1}\tilde\theta_n\|_{L^2}(\|\nabla u_{n-1}\|_{L^\infty}\|\Lambda^{s-2}\nabla\tilde\theta_n\|_{L^2}+\|\Lambda^{s-1}u_{n-1}\|_{L^6}\|\nabla\tilde\theta_n\|_{L^3})\\
&\le C\|\Lambda^{s-1}\tilde\theta_n\|_{L^2}(\|\Lambda^s\theta_{n-1}\|_{L^2}\|\Lambda^{s-1}\tilde\theta_n\|_{L^2}+\|\Lambda^s\theta_{n-1}\|_{L^2}\|\Lambda^{s-1}\tilde\theta_n\|_{L^2})\\
&\le 2C\|\Lambda^{s-1}\tilde\theta_n\|_{L^2}^2\|\Lambda^s\theta_{n-1}\|_{L^2},
\end{align*}
where we used \eqref{Commutator estimate} for $f=u_{n-1}$, $g=\nabla\tilde\theta_n$, $p=2$, $p_1=\infty$, $p_2=2$, $p_3=6$, $p_4=3$ and the assumption A$5_2$. 

On the other hand, using Proposition~2.1 in \cite{FRV12}, the term $\dis\intox\Lambda^{s-1}\tilde\theta_n\cdot\Lambda^{s-1}(\tilde u_{n-1}\cdot\nabla\theta_{n-1})$ can be estimated by
\begin{align*}
&\left|\intox\Lambda^{s-1}\tilde\theta_n\cdot\Lambda^{s-1}(\tilde u_{n-1}\cdot\nabla\theta_{n-1})\right|\\
&\le C\|\Lambda^{s-1}\tilde\theta_n\|_{L^2}(\|\tilde u_{n-1}\|_{L^\infty}\|\Lambda^{s-1}\nabla\theta_n\|_{L^2}+\|\Lambda^{s-1}\tilde u_{n-1}\|_{L^2}\|\nabla\theta_n\|_{L^\infty})\\
&\le C\|\Lambda^{s-1}\tilde\theta_n\|_{L^2}(\|\Lambda^{s-1}\tilde u_{n-1}\|_{L^2}\|\Lambda^{s}\theta_n\|_{L^2}+\|\Lambda^{s-1}\tilde\theta_{n-1}\|_{L^2}\|\Lambda^{s}\theta_n\|_{L^2})\\
&\le 2C\|\Lambda^{s-1}\tilde\theta_n\|_{L^2}\|\Lambda^{s-1}\tilde\theta_{n-1}\|_{L^2}\|\Lambda^{s}\theta_n\|_{L^2}
\end{align*}
Hence we deduce from \eqref{estimate on tilde theta n identity} that
\begin{align}\label{estimate on tilde theta 1}
\frac{1}{2}\frac{d}{dt}\|\Lambda^{s-1}\tilde\theta_n\|^2_{L^2}\le 2C\|\Lambda^{s-1}\tilde\theta_n\|_{L^2}(\|\Lambda^s\theta_{n-1}\|_{L^2}\|\Lambda^{s-1}\tilde\theta_n\|_{L^2}+\|\Lambda^{s-1}\tilde\theta_{n-1}\|_{L^2}\|\Lambda^s\theta_n\|_{L^2}).
\end{align}
Using \eqref{bound on theta n} on \eqref{estimate on tilde theta 1},
\begin{align}\label{estimate on tilde theta 2}
\frac{d}{dt}\|\Lambda^{s-1}\tilde\theta_n\|_{L^2}\le 4\sqrt{2}C(\|\Lambda^s\theta_0\|_{L^2}\|\Lambda^{s-1}\tilde\theta_n\|_{L^2}+\|\Lambda^{s-1}\tilde\theta_{n-1}\|_{L^2}\|\Lambda^s\theta_0\|_{L^2}).
\end{align}
Integrating \eqref{estimate on tilde theta 2} over $t$ and choosing $T$ small enough, we obtain
\begin{align}
\sup_{t\in[0,T]}\|\Lambda^{s-1}\tilde\theta_{n}(\cdot,t)\|_{L^2}\le\frac{1}{2}\sup_{t\in[0,T]}\|\Lambda^{s-1}\tilde\theta_{n-1}(\cdot,t)\|_{L^2}.
\end{align}
Thus $\theta_n$ is Cauchy in $L^\infty(0,T;H^{s-1})$ with $\theta_n$ converges strongly to $\theta$ in $L^\infty(0,T,H^{s-1})$. Since we assume that $s>\frac{d}{2}+1$, this also implies that the strong convergence occurs in a H\"{o}lder space relative to $x$ as $n\rightarrow\infty$, hence the limiting function $\theta$ is a solution of \eqref{abstract active scalar eqn nu=0}. Uniqueness of $\theta$ follows by the same argument given in \cite{FRV12} and we omit the details. It finishes the proof of Theorem~\ref{Local wellposedness in Sobolev space}. 
\end{step}
\end{proof}
\section{Convergence of solutions as $\nu\rightarrow0$}\label{Convergence of solutions}

In this section, we address the convergence of solutions to \eqref{abstract active scalar eqn} as $\nu\rightarrow0$ under the assumption A4 and give the proof of Theorem~\ref{Convergence of solutions as nu goes to 0}. Depending on the assumptions A$5_1$ and A$5_2$, we can address the convergence of solutions in two cases respectively:
\begin{itemize}
\item As we discussed before, it was proved in \cite{FV11b} that under the assumption A$5_1$, the equation \eqref{abstract active scalar eqn} for $\nu=0$ is ill-posed in the sense of Hadamard over $L^2$. Hence we focus on the case for Gevrey-class solutions $\theta^\nu$ to \eqref{abstract active scalar eqn}. By Theorem~\ref{Gevrey-class local wellposedness}, given Gevrey-class $s$ initial datum $\theta_0$ and forcing $S$, there exists $\bar{T},\bar{\tau}>0$ and a unique Gevrey-class solution $\theta^\nu$ to \eqref{abstract active scalar eqn} defined on $[0,\bar{T}]$ with radius of convergence at least $\bar{\tau}$ for all $\nu\ge0$. A natural question is the following: will the Gevrey-class solutions $\theta^\nu$ converge as $\nu\rightarrow0$? The answer is affirmative and is presented in Theorem~\ref{Convergence of solutions as nu goes to 0}, which shows that the Gevrey-class solutions $\theta^\nu$ converges to $\theta^0$ in some Gevrey-class norm as $\nu\rightarrow0$.
\item On the other hand, when assumption A$5_2$ is in force, by Theorem~\ref{Local wellposedness in Sobolev space}, the equation \eqref{abstract active scalar eqn} for $\nu=0$ is locally wellposed in Sobolev space $H^{s}$ for $s>\frac{d}{2}+1$. For sufficiently smooth initial data $\theta_0$ and forcing term $S$, we aim at showing that $\|(\theta^{\nu}-\theta^0)(\cdot,t)\|_{H^s}\rightarrow0$ as $\nu\rightarrow0$ for $s>\frac{d}{2}+1$ and $t\in[0,T]$. Such result is parallel to the one proved in \cite{FS18}, in which the authors proved that if $\theta^{\nu},\theta^0$ are $C^\infty$ smooth classical solutions of the diffusive system \eqref{abstract active scalar eqn diffusive} for $\nu>0$ and $\nu=0$ respectively with initial datum $\theta_0\in L^2$ and forcing term $S\in C^\infty$, then $\|(\theta^{\nu}-\theta^0)(\cdot,t)\|_{H^s}\rightarrow0$ as $\nu\rightarrow0$ for $s\ge0$ and $t>0$.
\end{itemize}
\begin{remark}
In the diffusive system \eqref{abstract active scalar eqn diffusive} studied in \cite{FS18} there is no smoothing assumption imposed on $\{T_{ij}^{\nu}\}_{\nu\ge0}$ when $\nu >0$. The main reason for the difference is that the diffusive term $\kappa\Delta\theta^\nu$ present in \eqref{abstract active scalar eqn diffusive} is sufficient to smooth out the solution $\theta^\nu$ for all $\nu\ge0$.
\end{remark}
\begin{proof}[\bf Proof of Theorem~\ref{Convergence of solutions as nu goes to 0}]
We divide the proof into two cases:
\setcounter{case}{0}
\begin{case}
When A5$_1$ is in force. Fix $s\ge1$ and $r>\frac{d}{2}+\frac{3}{2}$. Throughout this proof, $C>0$ is a generic constant which depends on $C_0, \theta_0,S$, $s,r,d,\bar{T},\bar{\tau}$ and is independent of $\nu$.  Let $\theta^\nu,\theta^0$ be the Gevrey-class $s$ solutions to \eqref{abstract active scalar eqn} on $[0,\bar{T}]$ as obtained in Theorem~\ref{Gevrey-class local wellposedness}. We define $\phi^\nu=\theta^\nu-\theta^0$ and write
\begin{equation*}
\|\phi^\nu\|_{\tau,r}^2=\|\Lambda^re^{\tau\Lambda^\frac{1}{s}}\phi^\nu\|_{L^2}^2=\sum_{k\in\mathbb{Z}^d_*}|k|^{2r}e^{2\tau|k|^\frac{1}{s}}|\widehat{\phi^\nu}(k)|^2.
\end{equation*}
Then $\phi^\nu$ satisfies the following equation on $[0,\bar{T}]$
\begin{align}\label{eqn of phi}
\partial_t \phi^\nu+(u^\nu-u^0)\cdot\nabla\theta^0+u^\nu\cdot\nabla\phi^\nu=0,
\end{align}
where $u^0_j=\partial_{x_i} T_{ij}^{\nu}[\theta^0]$ for all $i,j$. From \eqref{eqn of phi}, we have the {\it a priori} estimate
\begin{align}\label{a priori estimate on phi in analytic norm}
\frac{1}{2}\frac{d}{dt}\|\phi^\nu\|^2_{\tau,r}&=\dot{\tau}\|\Lambda^\frac{1}{2s}\phi^\nu\|^2_{\tau,r}+\langle (u^\nu-u^0)\cdot\nabla\theta,(-\Delta)^re^{2\tau(-\Delta)^\frac{1}{2}}\phi^\nu\rangle+\langle u^\nu\cdot\nabla\phi^\nu,(-\Delta)^re^{2\tau(-\Delta)^\frac{1}{2}}\phi^\nu\rangle\notag\\
&=\dot{\tau}\|\Lambda^\frac{1}{2s}\phi^\nu\|^2_{\tau,r}+\mathcal{R}_1+\mathcal{R}_2.
\end{align}
Using Plancherel's theorem, the nonlinear term $\mathcal{R}_1$ can be written as
\begin{align*}
\mathcal{R}_1=i(2\pi)^d\sum_{j+k=l;j,k,l\in\mathbb{Z}_*^d}\widehat{(u^\nu-u^0)}(j)\cdot k\widehat{\theta^0}(k)|l|^{2r}e^{2r|l|}\hat{\phi^\nu}(-l).
\end{align*}
The term $\mathcal{R}_1$ can be estimated as follows.
\begin{align}\label{estimates on R_3}
\mathcal{R}_1&\le C\sum_{j+k=l;j,k,l\in\mathbb{Z}_*^d}|j||k|(|j|^r+|k|^r)|\hat{\phi^\nu}|e^{\tau|j|}|\widehat{\theta^0}(k)|e^{\tau|k|}|l|^r|\hat{\phi^\nu}(l)|e^{\tau|l|}\notag\\
&\qquad+C\sum_{j+k=l;j,k,l\in\mathbb{Z}_*^d}|(\widehat{T^\nu}-\widehat{T^0})(j)||\widehat{\theta^0}(j)||j||k|(|j|^r+|k|^r)e^{\tau|j|}|\widehat{\theta^0}(k)|e^{\tau|k|}|l|^r|\hat{\phi^\nu}(l)|e^{\tau|l|}\notag\\
&\le C\|\Lambda^\frac{1}{2s}\theta^0\|_{\tau,r}\|\Lambda^\frac{1}{2s}\phi^\nu\|_{\tau,r}\|\phi^\nu\|_{\tau,r}\notag\\
&\qquad+C\|\Lambda^\frac{1}{2s}\theta^0\|_{\tau,r}\|\Lambda^\frac{1}{2s}\phi^\nu\|_{\tau,r}\sum_{j\in\mathbb{Z}_*^d}|j|^\frac{3}{2}|\widehat{\theta^0}(j)|e^{\tau|j|}|(\widehat{T^\nu}-\widehat{T^0})(j)|\notag\\
&\le C\|\Lambda^\frac{1}{2s}\theta^0\|_{\tau,r}\|\Lambda^\frac{1}{2s}\phi^\nu\|_{\tau,r}\|\phi^\nu\|_{\tau,r}\notag\\
&\qquad+C\|\Lambda^\frac{1}{2s}\theta^0\|_{\tau,r}\|\Lambda^\frac{1}{2s}\phi^\nu\|_{\tau,r}\left(\sum_{j\in\mathbb{Z}_*^d}|j|^{d+3}|\widehat{\theta^0}(j)|^2e^{2\tau|j|}|(\widehat{T^\nu}-\widehat{T^0})(j)|^2\right)^\frac{1}{2}\left(\sum_{j\in\mathbb{Z}_*^d}|j|^{-d}\right)^\frac{1}{2}\notag\\
&\le C\|\Lambda^\frac{1}{2s}\theta^0\|_{\tau,r}\|\Lambda^\frac{1}{2s}\phi^\nu\|_{\tau,r}\|\phi^\nu\|_{\tau,r}\notag\\
&\qquad+C\|\Lambda^\frac{1}{2s}\theta^0\|_{\tau,r}\|\Lambda^\frac{1}{2s}\phi^\nu\|_{\tau,r}\left(\sum_{j\in\mathbb{Z}_*^d}|j|^{d+3}|\widehat{\theta^0}(j)|^2e^{2\tau|j|}|(\widehat{T^\nu}-\widehat{T^0})(j)|^2\right)^\frac{1}{2}
\end{align}
where the last inequality holds since $d\ge2$ and $\left(\sum_{j\in\mathbb{Z}_*^d}|j|^{-d}\right)^\frac{1}{2}<\infty$. Similarly, $\mathcal{R}_2$ can be estimated by
\begin{align}\label{estimates on R_4}
\mathcal{R}_2\le C\|\Lambda^\frac{1}{2s}\phi^\nu\|_{\tau,r}^2\|\theta^\nu\|_{\tau,r}.
\end{align}
Using the bounds \eqref{estimates on R_3} and \eqref{estimates on R_4} on \eqref{a priori estimate on phi in analytic norm}, we obtain
\begin{align*}
\frac{1}{2}\frac{d}{dt}\|\phi^\nu\|^2_{\tau,r}&\le\left(\dot{\tau}+C\|\theta^\nu\|_{\tau,r}+C\|\Lambda^\frac{1}{2s}\theta^0\|_{\tau,r}^2\right)\|\Lambda^\frac{1}{2s}\phi^\nu\|_{\tau,r}^2\\
&\qquad+C\|\phi^\nu\|_{\tau,r}^2+C\sum_{j\in\mathbb{Z}_*^d}|j|^{d+3}|\widehat{\theta^0}(j)|^2e^{2\tau|j|}|(\widehat{T^\nu}-\widehat{T^0})(j)|^2.
\end{align*}
Choose $\tau=\tau(t)\le\bar{\tau}$ such that
\begin{align*}
\left\{ \begin{array}{l}
\dot{\tau}+C\|\theta^\nu\|_{\tau,r}+C\|\Lambda^\frac{1}{2s}\theta^0\|_{\tau,r}^2<0, \\
\tau<\bar{\tau},
\end{array}\right.
\end{align*}
then using the bounds \eqref{bound on solution 2} and \eqref{bound on solution 1}, there exists $T<\bar{T}$ such that for $t\in[0,T]$, we have
\begin{align*}
\frac{d}{dt}\|\phi^\nu\|_{\tau,r}^2\le C\|\phi^\nu\|_{\tau,r}^2+C\sum_{j\in\mathbb{Z}_*^d}|j|^{d+3}|\widehat{\theta^0}(j)|^2e^{2\tau|j|}|(\widehat{T^\nu}-\widehat{T^0})(j)|^2.
\end{align*}
Integrating the above with respect to $t$, for $t\in[0,T]$, we obtain
\begin{align*}
\|\phi^\nu(\cdot,t)\|_{\tau,r}^2\le e^{CT}C\sum_{j\in\mathbb{Z}_*^d}|j|^{d+3}|\widehat{\theta^0}(j)|^2e^{2\tau|j|}|(\widehat{T^\nu}-\widehat{T^0})(j)|^2.
\end{align*}
Since $\|\theta^0\|_{\tau,r}<\infty$ with $r>\frac{d}{2}+\frac{3}{2}$, it implies $\sum_{j\in\mathbb{Z}_*^d}|j|^{d+3}|\widehat{\theta^0}(j)|^2e^{2\tau|j|}<\infty$, and hence by the assumption A4,
\begin{align*}
\lim_{\nu\rightarrow0}\sum_{j\in\mathbb{Z}_*^d}|j|^{d+3}|\widehat{\theta^0}(j)|^2e^{2\tau|j|}|(\widehat{T^\nu}-\widehat{T^0})(j)|^2=0.
\end{align*}
Therefore the result \eqref{convergence} follows.
\end{case}
\begin{case} When A5$_2$ is in force. Fix $s>\frac{d}{2}+1$, let $\theta^\nu,\theta^0$ be the $H^s$ to \eqref{abstract active scalar eqn} on $[0,T]$ as obtained in Theorem~\ref{Local wellposedness in Sobolev space}. We define $\phi^\nu=\theta^{\nu}-\theta^0$, then $\phi^\nu$ satisfies \eqref{eqn of phi} on $[0,T]$. 

We first show that, for $t\in[0,T]$,
\begin{align}\label{L2 convergence of phi}
\lim_{\nu\rightarrow0}\|\phi^\nu(\cdot,t)\|_{L^2}=0.
\end{align}
Following the proof of Theorem~\ref{Local wellposedness in Sobolev space}, shrinking the time $T$ if necessary, there exists $C=C(T,\theta_0,S)>0$ independent of $\nu$ such that, for all $\nu\ge0$,
\begin{align}\label{uniform bound on theta}
\sup_{0\le t\le T}\|\theta^\nu(\cdot,t)\|_{H^s}\le C.
\end{align}
We multiply \eqref{eqn of phi} by $\phi^\nu$ and integrate, for $t\in[0,T]$,
\begin{align}\label{integral eqn for phi}
\frac{1}{2}\frac{d}{dt}\|\phi^\nu(\cdot,t)\|_{L^2}^2=-\int(u^{\nu}-u^0)\cdot\nabla\theta^0\cdot\phi^\nu(x,t)dx.
\end{align}
We estimate the right side of \eqref{integral eqn for phi} as follows. Using Sobolev embedding theorem and the bound \eqref{uniform bound on theta},
\begin{align}\label{estimate on the cross term for phi}
\left|-\int(u^{\nu}-u^0)\cdot\nabla\theta^0\cdot\phi^\nu(x,t)dx\right|&\le\|(u^{\nu}-u^0)(\cdot,t)\|_{L^2}\|\phi^\nu(\cdot,t)\|_{L^2}\|\nabla\theta^0(\cdot,t)\|_{L^\infty}\notag\\
&\le C\|(u^{\nu}-u^0)(\cdot,t)\|_{L^2}\|\phi^\nu(\cdot,t)\|_{L^2}\notag\\
&\le\frac{C}{2}\|(u^{\nu}-u^0)(\cdot,t)\|_{L^2}^2+\frac{C}{2}\|\phi^\nu(\cdot,t)\|_{L^2}^2.
\end{align}
We focus on the term $\|(u^{\nu}-u)(\cdot,t)\|^2_{L^2}$ as in \eqref{estimate on the cross term for phi}. Using Plancherel Theorem and assumption A$5_2$, for each $j$,
\begin{align*}
\|(u_j^{\nu}-u_j)(\cdot,t)\|^2_{L^2}&=\sum_{k\in\mathbb{Z}^d}|\widehat{(u_j^{\nu}-u_j)}(k,t)|^2\notag\\
&=\sum_{k\in\mathbb{Z}^d}|(\widehat{\partial_{x_i} T_{ij}^{\nu}}\widehat{\theta^{\nu}}-\widehat{\partial_{x_i} T^0_{ij}}\widehat{\theta^0})(k,t)|^2\notag\\
&\le\sum_{k\in\mathbb{Z}^d}|\widehat{\partial_{x_i} T_{ij}^{\nu}}|^2|\widehat\phi|^2(k,t)+\sum_{k\in\mathbb{Z}^3}|\widehat{T_{ij}^{\nu}}-\widehat{T^0_{ij}}|^2|\widehat{\nabla\theta^0}|^2(k,t)\notag\\
&\le C_0\|\phi^\nu(\cdot,t)\|_{L^2}^2+I(\nu,t),
\end{align*}
where $I(\nu,t)=\sum_{k\in\mathbb{Z}^3}|\widehat{T_{ij}^{\nu}}-\widehat{T^0_{ij}}|^2|\widehat{\nabla\theta^0}|^2(k,t)$. Applying the above estimate on \eqref{estimate on the cross term for phi}, we obtain 
\begin{equation*}
\frac{1}{2}\frac{d}{dt}\|\phi^\nu(\cdot,t)\|_{L^2}^2\le\frac{C}{2}\Big(C_0\|\phi^\nu(\cdot,t)\|_{L^2}^2+I(\nu,t)\Big)+\frac{C}{2}\|\phi^\nu(\cdot,t)\|_{L^2}^2.
\end{equation*}
For $t\in[0,T]$, since $\|\theta^0(\cdot,t)\|_{L^2}<\infty$, by assumption A4, we have $\lim_{\nu\rightarrow0}I(\nu,t)=0$. Hence taking $\nu\rightarrow0$ and using Gr\"{o}nwall's inequality, we conclude that \eqref{L2 convergence of phi} holds for $t\in[0,T]$.

Finally, we apply the Gagliardo-Nirenberg interpolation inequality and the bound \eqref{uniform bound on theta} to obtain, for $t\in[0,T]$,
\begin{align*}
\|(\theta^{\nu}-\theta^0)(\cdot,t)\|_{ H^{s-1}}&\le C(d)\|(\theta^{\nu}-\theta^0)(\cdot,t)\|_{L^2}^{\gamma}\|(\theta^{\nu}-\theta^0)(\cdot,t)\|_{ H^s}^{1-\gamma}\\
&\le C(d)C^{1-\gamma}\|(\theta^{\nu}-\theta^0)(\cdot,t)\|_{L^2}^{\gamma},
\end{align*}
where $\gamma\in(0,1)$ depends on $s$ and $C(d)>0$ is a positive constant which depends on $d$ but is independent of $\nu$. By taking $\nu\rightarrow0$ and applying the $L^2$-convergence \eqref{L2 convergence of phi} just proved, we conclude that \eqref{convergence sobolev} holds for $t\in[0,T]$ as well.
\end{case}
\end{proof}
\section{Applications to physical models}\label{Applications to physical models}
We now apply our results claimed in Section~\ref{main results} to some physical models, namely the magnetogeostrophic (MG) equations and the incompressible porous media (IPMB) equations discussed in Section~\ref{introduction}.
\subsection{Magnetogeostrophic equations}
We first consider the following magnetogeostrophic (MG) equation in the domain $\mathbb{T}^3\times(0,\infty)=[0,2\pi]^3\times(0,\infty)$ with periodic boundary conditions:
\begin{align}
\label{MG active scalar} \left\{ \begin{array}{l}
\partial_t\theta^{\nu}+u^\nu\cdot\nabla\theta^{\nu}=S, \\
u=M^{\nu}[\theta^{\nu}],\theta(x,0)=\theta_0(x)
\end{array}\right.
\end{align}
via a Fourier multiplier operator $M^{\nu}$ which relates $u^\nu$ and $\theta^{\nu}$. More precisely,
\begin{align*}
u^\nu_j=M^{\nu}_j [\theta^{\nu}]=(\widehat{M^{\nu}_j}\hat\theta^{\nu})^\vee
\end{align*}
for $j\in\{1,2,3\}$. The explicit expression for the components of $\widehat M^{\nu}$ as functions of the Fourier variable $k=(k_1,k_2,k_3)\in\Z^3$ are given by \eqref{MG Fourier symbol_1}-\eqref{MG Fourier symbol_4} in Section~\ref{introduction}. We write $M^\nu_j=\partial_iT_{ij}^{\nu}$ for convenience. To apply the results from Section~\ref{main results}, it suffices to show that the sequence of operators $\{T^{\nu}_{ij}\}_{\nu\ge0}$ satisfy the assumptions A1--A4 and A$5_1$ given in Section~\ref{introduction}. We first prove the following lemma for the MG equations.
\begin{lemma}\label{convergence of operator}
For each $L>0$, 
\begin{align}\label{stronger assumption A4}
\lim_{\nu\rightarrow0}\sup_{\{k\in\mathbb{Z}^3:k\neq0,|k|\le L\}}\frac{|\widehat M^{\nu}(k)-\widehat M^0(k)|}{|k|}=0.
\end{align}
\end{lemma}
\begin{proof}[\bf Proof]
We only give the details for $\widehat M^{\nu}_1$, since the cases for $\widehat M^{\nu}_2$ and $\widehat M^{\nu}_3$ are similar. We fix $L>0$, then for each $k\in\mathbb{Z}^3\backslash(\{k=0\}$ with $|k|\le L$, we have
\begin{align*}
&\frac{|\widehat M_1^{\nu}(k)-\widehat M^0_1(k)|}{|k|}\\
&=\frac{|-\nu k_1k_3^3|k|^6+\nu k_1 k_2^4k_3|k|^4-\nu^2 k_2k_3|k|^{10}+\nu^2 k_1k_2^2k_3|k|^8-2\nu k_2^3 k_3|k|^6|}{(|k|^2k_3^2+\nu^2|k|^8+2\nu|k|^4k_2^2+k_2^4)(k_3^2|k|^2+k_2^4)|k|}.\\
&\le\frac{\nu|k_1||k_3|^3|k|^6}{|k|^{5}k_3^4}+\frac{\nu|k_1|k_2^4|k_3||k|^4}{|k|^{5}k_3^4}+\frac{\nu^2|k_2||k_3||k|^{10}}{|k|^{5}k_3^4}+\frac{\nu^2|k_1|k_2^2|k_3|k|^8}{|k|^{5}k_3^4}+\frac{2\nu|k_2|^3|k_3||k|^6}{|k|^{5}k_3^4}\\
&\le\nu L^{10}+\nu L^{10}+\nu^2 L^{12}+\nu^2 L^{12}+2\nu L^{10}.
\end{align*}
Hence 
\begin{align*}
\lim_{\nu\rightarrow0}\sup_{\{k\in\mathbb{Z}^3:k\neq0,|k|\le L\}}\frac{|\widehat M_1^{\nu}(k)-\widehat M^0_1(k)|}{|k|}=0.
\end{align*}
\end{proof}
\begin{proposition}\label{assumption checked MG}
Let $M^\nu_j=\partial_iT_{ij}^{\nu}$, where $M^\nu$ is given by \eqref{MG Fourier symbol_1}-\eqref{MG Fourier symbol_3}. Then $T_{ij}^{\nu}$ satisfy the assumptions A1--A4 and A$5_1$ given in Section~\ref{introduction}.
\end{proposition}
\begin{proof}[\bf Proof]
The details for the proof can be found in \cite{FS18} Lemma~5.1--5.2 and from the discussion in (\cite{FV11a}, Section 4). For example, to show $T_{ij}^{\nu}$ satisfy the assumption A3, we only give the details for $\widehat M^{\nu}_1$ since the cases for $\widehat M^{\nu}_2$ and $\widehat M^{\nu}_3$ are almost identical. We fix $\nu\in(0,1]$ and consider the following cases:
\setcounter{case}{0}
\begin{case} $|k|>\nu^{-\frac{1}{2}}$. Then for each $k\in\mathbb{Z}^3/\{k=0\}$, 
\begin{align*}
\frac{|\widehat M_1^{\nu}(k)|}{|k|}=\frac{|k_2k_3|k|^2-k_1k_3(k_2^2+\nu|k|^4)|}{|k|(|k|^2k_3^2+(k_2^2+\nu|k|^4)^2)}.
\end{align*}
Since $k\neq0$, so $|k|\ge|k_j|\ge1$ for $j=1,2,3$, in particular $|k|^{-1}<\nu^\frac{1}{2}$. Hence we obtain
\begin{align*}
\frac{|\widehat M_1^{\nu}(k)|}{|k|}&\le\frac{|k_2k_3||k|^2}{|k|^3k_3^2}+\frac{|k_1k_3|k_2^2}{|k|^3k_3^2}+\frac{\nu|k_1k_3||k|^4}{\nu^2|k|^8}\\
&\le\frac{1}{|k_3|}+\frac{1}{|k_3|}+\frac{1}{\nu|k|^2}\\
&\le 2+\frac{\nu}{\nu}=3.
\end{align*}
\end{case}
\begin{case} $|k|\le\nu^{-\frac{1}{2}}$. Then for each $k\in\mathbb{Z}^3/\{k=0\}$,
\begin{align*}
\frac{|\widehat M_1^{\nu}(k)|}{|k|}&\le\frac{|k_2k_3||k|^2}{|k|^3k_3^2}+\frac{|k_1k_3|k_2^2}{|k|^3k_3^2}+\frac{\nu|k_1k_3||k|^4}{|k|^3k_3^2}\\
&\le\frac{1}{|k_3|}+\frac{1}{|k_3|}+\frac{\nu|k|^2}{|k_3|}\\
&\le2+\nu\cdot(\nu^{-\frac{1}{2}})^2=3.
\end{align*}
Combining two cases, we have
\begin{align*}
\sup_{\nu\in(0,1]}\sup_{\{k\in\mathbb{Z}^3:k\neq0\}}\frac{|\widehat M_1^{\nu}(k)|}{|k|}\le3,
\end{align*}
and hence assumption A3 holds for some $C_0>0$ independent of $\nu$, which means that
\begin{equation}\label{bounds on operator M for MG}
\sup_{\nu\in(0,1]}\sup_{\{k\in\mathbb{Z}^3:k\neq0\}}\frac{|\widehat M_1^{\nu}(k)|}{|k|}\le C_0.
\end{equation}
On the other hand, to show $T_{ij}^{\nu}$ satisfy the assumption A4, Fix $g$ with $\|g\|_{L^2}<\infty$ and we claim that
\begin{equation}\label{checking of A4 MG 1}
\lim_{\nu\rightarrow0}\sum_{k\in\mathbb{Z}^d}|\widehat{T^{\nu}_{ij}}(k)-\widehat{T^0_{ij}}(k)|^2|\widehat{g}(k)|^2=0.
\end{equation}
Let $\varepsilon>0$ be given. Then $\dis\sum_{k\in\mathbb{Z}^3}|\widehat{g}(k)|^2<\infty$, so there exists $L=L(\varepsilon)>0$ such that $\dis\sum_{k\in\mathbb{Z}^3, |k|>L}|\widehat{g}(k)|^2<\varepsilon$. Hence for $1\le i,j\le d$, we have
\begin{align}
&\sum_{k\in\mathbb{Z}^d}|\widehat{T^{\nu}_{ij}}(k)-\widehat{T^0_{ij}}(k)|^2|\widehat{g}(k)|^2\notag\\
&\le\sum_{k\in\mathbb{Z}^3:k\neq0}\frac{|\widehat M^{\nu}(k)-\widehat M^0(k)|^2|\widehat{g}(k)|^2}{|k|^2}\notag\\
&=\sum_{k\in\mathbb{Z}^3:k\neq0,|k|\le L}\frac{|\widehat M^{\nu}(k)-\widehat M^0(k)|^2|\widehat{g}(k)|^2}{|k|^2}+\sum_{k\in\mathbb{Z}^3:k\neq0,|k|>L}\frac{(|\widehat M^{\nu}(k)|^2+|\widehat M^0(k)|^2)|\widehat{g}(k)|^2}{|k|^2}\notag\\
&\le\left(\sup_{\{k\in\mathbb{Z}^3:k\neq0,|k|\le L\}}\frac{|\widehat M^{\nu}(k)-\widehat M^0(k)|}{|k|}\right)^2\|g\|_{L^2}^2+2C_0^2\varepsilon.\label{checking of A4 MG 2}
\end{align}
where the last inequality follows by the bound \eqref{bounds on operator M for MG}. Using \eqref{stronger assumption A4} in Lemma~\ref{convergence of operator} and taking $\nu\rightarrow0$ on \eqref{checking of A4 MG 2},
\begin{align*}
\lim_{\nu\rightarrow0}\sum_{k\in\mathbb{Z}^d}|\widehat{T^{\nu}_{ij}}(k)-\widehat{T^0_{ij}}(k)|^2|\widehat{g}(k)|^2\le 2C_0^2\varepsilon.
\end{align*}
Since $\varepsilon>0$ is arbitrary, \eqref{checking of A4 MG 1} follows and therefore $T_{ij}^{\nu}$ satisfy the assumption A4.
\end{case}
\end{proof} 
In view of Proposition~\ref{assumption checked MG}, the abstract Theorem~\ref{Wellposedness in Sobolev space}-\ref{Gevrey-class local wellposedness} and Theorem~\ref{Convergence of solutions as nu goes to 0} may therefore be applied to the MG equations \eqref{MG active scalar} in order to obtain the wellposedness and convergence of Gevrey-class solutions. More precisely, we have
\begin{theorem}[Wellposedness in Sobolev space for the MG equations]\label{Wellposedness in Sobolev space MG}
Let $\theta_0\in W^{s,3}$ for $s\ge0$ and $S$ be a $C^\infty$-smooth source term. Then for each $\nu>0$, we have:
\begin{itemize}
\item if $s=0$, there exists unique global weak solution to \eqref{MG active scalar} such that
\begin{align*}
\theta^\nu&\in BC((0,\infty);L^3),\\
u^\nu&\in C((0,\infty);W^{2,3}).
\end{align*}
In particular, $\theta^\nu(\cdot,t)\rightarrow\theta_0$ weakly in $L^3$ as $t\rightarrow0^+$.
\item if $s>0$, there exists a unique global-in-time solution $\theta^\nu$ to \eqref{MG active scalar} such that $\theta^\nu(\cdot,t)\in W^{s,3}$ for all $t\ge0$. Furthermore, for $s=1$, we have the following single exponential growth in time on $\|\nabla\theta^\nu(\cdot,t)\|_{L^3}$:
\begin{align*}
\|\nabla\theta^\nu(\cdot,t)\|_{L^3}\le C \|\nabla\theta_0\|_{L^3}\exp\left(C\left(t\|\theta_0\|_{W^{1,3}}+t^2\|S\|_{L^\infty}+t\|S\|_{W^{1,3}}\right)\right),
\end{align*}
where $C>0$ is a constant which depends only on some dimensional constants.
\end{itemize}
\end{theorem}
\begin{theorem}[Gevrey-class global wellposedness for the MG equations]\label{Gevrey-class global wellposedness MG}
Fix $s\ge1$. Let $\theta_0$ and $S$ be of Gevrey-class $s$ with radius of convergence $\tau_0>0$. Then for each $\nu>0$, there exists a unique Gevrey-class $s$ solution $\theta^\nu$ to \eqref{MG active scalar} on $\mathbb{T}^3\times[0,\infty)$ with radius of convergence at least $\tau=\tau(t)$ for all $t\in[0,\infty)$, where $\tau$ is a decreasing function satisfying
\begin{align*}
\tau(t)\ge\tau_0e^{-C\left(\|e^{\tau_0\Lambda^\frac{1}{s}}\theta_0\|_{L^2}+2\|e^{\tau_0\Lambda^\frac{1}{s}}S\|_{L^2}\right)t}.
\end{align*}
Here $C>0$ is a constant which depends on $\nu$ but independent of $t$.
\end{theorem}
\begin{theorem}[Gevrey-class local wellposedness for the MG equations]\label{Gevrey-class local wellposedness MG}
Fix $s\ge1$, $r>3$ and $K_0>0$. Let $\theta_0$ and $S$ be of Gevrey-class $s$ with radius of convergence $\tau_0>0$ and
\begin{align*}
\|\Lambda^re^{\tau_0\Lambda^\frac{1}{s}}\theta^0(\cdot,0)\|_{L^2}\le K_0,\qquad\|\Lambda^re^{\tau_0\Lambda^\frac{1}{s}}S\|_{L^2}\le K_0.
\end{align*}
There exists $\bar{T},\bar{\tau}>0$ and a unique Gevrey-class $s$ solution $\theta^0$ to \eqref{MG active scalar} for $\nu=0$ defined on $\mathbb{T}^3\times[0,\bar{T}]$ with radius of convergence at least $\bar{\tau}$. Moreover, there exists a constant $C=C(K_0)>0$ independent of $\nu$ such that for all $t\in[0,\bar{T}]$,
\begin{align*}
\|\Lambda^re^{\bar{\tau}\Lambda^\frac{1}{s}}\theta^\nu(\cdot,t)\|_{L^2}&\le C,\,\forall\nu>0,\\
\|\Lambda^re^{\bar{\tau}\Lambda^\frac{1}{s}}\theta^0(\cdot,t)\|_{L^2}&\le C.
\end{align*}
Here $\theta^\nu$ are Gevrey-class $s$ solutions to \eqref{MG active scalar} for $\nu>0$ as described in Theorem~\ref{Gevrey-class global wellposedness MG}.
\end{theorem}
\begin{theorem}[Convergence of solutions as $\nu\rightarrow0$ for the MG equations]\label{Convergence of solutions as nu goes to 0 MG}
Fix $s\ge1$, $r>3$ and $K_0>0$. Let $\theta_0$ and $S$ be of Gevrey-class $s$ with radius of convergence $\tau_0>0$ and satisfy the assumptions given in Theorem~\ref{Gevrey-class local wellposedness MG}. If $\theta^\nu$ and $\theta^0$ are Gevrey-class $s$ solutions to \eqref{MG active scalar} for $\nu>0$ and $\nu=0$ respectively with initial datum $\theta_0$ on $\mathbb{T}^3\times[0,\bar{T}]$ with radius of convergence at least $\bar{\tau}$ as described in Theorem~\ref{Gevrey-class local wellposedness MG}, then there exists $T<\bar{T}$ and $\tau=\tau(t)<\bar{\tau}$ such that, for $t\in[0,T]$, we have
\begin{align*}
\lim_{\nu\rightarrow0}\|(\Lambda^re^{\tau\Lambda^\frac{1}{s}}\theta^{\nu}-\Lambda^re^{\tau\Lambda^\frac{1}{s}}\theta^0)(\cdot,t)\|_{L^2}=0.
\end{align*}
\end{theorem}
%
%
\subsection{Incompressible porous media equation}\label{Incompressible porous media equation}
Next we study the incompressible porous media Brinkmann (IPMB) equation. Specifically, we address the following active scalar equation in $\mathbb{T}^2\times[0,\infty)$ with periodic boundary conditions:
\begin{align}
\label{IPMB}
\left\{ \begin{array}{l}
\partial_t\theta^\nu+(u^\nu\cdot\nabla)\theta^\nu=0, \\
u^\nu=M^{\nu}[\theta^{\nu}],\theta^\nu(x,0)=\theta_0(x),
\end{array}\right.
\end{align}
where the symbol of $M^\nu$ is given by \eqref{Fourier multiplier symbol IPMB} with
\begin{align}
\widehat M^{\nu}_1(k)&=\frac{1}{1+\nu(k_1^2+k_2^2)}\left(\frac{k_1k_2}{k_1^2+k_2^2}\right),\label{IPMB Fourier symbol_1}\\
\widehat M^{\nu}_2(k)&=\frac{1}{1+\nu(k_1^2+k_2^2)}\left(\frac{-k_1^2}{k_1^2+k_2^2}\right).\label{IPMB Fourier symbol_2}
\end{align}
We also write $M^\nu_j=\partial_iT_{ij}^{\nu}$ for convenience. To apply the results from Section~\ref{main results}, it suffices to show that the sequence of operators $\{T^{\nu}_{ij}\}_{\nu\ge0}$ satisfy the assumptions A1--A4 and A$5_2$ given in Section~\ref{introduction}. 
\begin{proposition}\label{assumption checked IPMB}
Let $M^\nu_j=\partial_iT_{ij}^{\nu}$, where $M^\nu$ is given by \eqref{IPMB Fourier symbol_1}-\eqref{IPMB Fourier symbol_2}. Then $T_{ij}^{\nu}$ satisfy the assumptions A1--A4 and A5$_2$ given in Section~\ref{introduction}.
\end{proposition}
\begin{proof}[\bf Proof]
It suffices to check that $T_{ij}^{\nu}$ satisfy assumptions A3 and A4. To show that $T_{ij}^{\nu}$ satisfy A3, for each $k\in\mathbb{Z}^2/\{k=0\}$,
\begin{align*}
\frac{|\widehat M_1^{\nu}(k)|}{|k|}&=\frac{|k_1k_2|}{1+\nu|k|^2}\times\frac{1}{|k|^3}\le1,
\end{align*}
since $|k|\ge1$. Similarly, $\dis\frac{|\widehat M_2^{\nu}(k)|}{|k|}\le1$. And to see that $T_{ij}^{\nu}$ satisfies A$5_2$, similar to the case of MG equation, it suffice to show that for each $L>0$, 
\begin{equation}\label{stronger assumption A4 IPMB}
\lim_{\nu\rightarrow0}\sup_{\{k\in\mathbb{Z}^3:k\neq0,|k|\le L\}}\frac{|\widehat M^{\nu}(k)-\widehat M^0(k)|}{|k|}=0.
\end{equation}
Fix $L>0$ and for each $k\in\mathbb{Z}^2/\{k=0\}$ with $|k|\le L$, we have
\begin{align*}
\frac{|\widehat M_1^{\nu}(k)-\widehat M^0_1(k)|}{|k|}&=\left|\frac{\nu|k|^2}{(1+\nu|k|^2)}\right|\times\left|\frac{k_1k_2}{|k|^3}\right|\\
&\le\nu\times\frac{|k_1k_2|}{|k|}\le\nu L,
\end{align*}
hence 
\begin{align*}
\lim_{\nu\rightarrow0}\sup_{\{k\in\mathbb{Z}^2:k\neq0,|k|\le L\}}\frac{|\widehat M_1^{\nu}(k)-\widehat M^0_1(k)|}{|k|}=0.
\end{align*}
By the same argument, we also have $\dis\lim_{\nu\rightarrow0}\sup_{\{k\in\mathbb{Z}^2:k\neq0,|k|\le L\}}\frac{|\widehat M_2^{\nu}(k)-\widehat M^0_2(k)|}{|k|}=0$ and \eqref{stronger assumption A4 IPMB} follows.
\end{proof}
Thanks to Proposition~\ref{assumption checked IPMB}, the abstract Theorem~\ref{Wellposedness in Sobolev space}-\ref{Convergence of solutions as nu goes to 0} can be applied to the IPMB equations \eqref{IPMB}. More precisely, we have
\begin{theorem}[Wellposedness in Sobolev space for the IPMB equations]\label{Wellposedness in Sobolev space IPMB}
Let $\theta_0\in W^{s,2}$ for $s\ge0$. Then for each $\nu>0$, we have:
\begin{itemize}
\item if $s=0$, there exists unique global weak solution to \eqref{MG active scalar} such that
\begin{align*}
\theta^\nu&\in BC((0,\infty);L^2),\\
u^\nu&\in C((0,\infty);W^{2,2}).
\end{align*}
In particular, $\theta^\nu(\cdot,t)\rightarrow\theta_0$ weakly in $L^2$ as $t\rightarrow0^+$.
\item if $s>0$, there exists a unique global-in-time solution $\theta^\nu$ to \eqref{MG active scalar} such that $\theta^\nu(\cdot,t)\in W^{s,2}$ for all $t\ge0$. Furthermore, for $s=1$, we have the following single exponential growth in time on $\|\nabla\theta^\nu(\cdot,t)\|_{L^2}$:
\begin{align*}
\|\nabla\theta^\nu(\cdot,t)\|_{L^2}\le C \|\nabla\theta_0\|_{L^2}\exp\left(Ct\|\theta_0\|_{W^{1,2}}\right),
\end{align*}
where $C>0$ is a constant which depends only on some dimensional constants.
\end{itemize}
\end{theorem}
\begin{theorem}[Gevrey-class global wellposedness for the IPMB equations]\label{Gevrey-class global wellposedness IPMB}
Fix $s\ge1$. Let $\theta_0$ be of Gevrey-class $s$ with radius of convergence $\tau_0>0$. Then for each $\nu>0$, there exists a unique Gevrey-class $s$ solution $\theta^\nu$ to \eqref{IPMB} on $\mathbb{T}^2\times[0,\infty)$ with radius of convergence at least $\tau=\tau(t)$ for all $t\in[0,\infty)$, where $\tau$ is a decreasing function satisfying
\begin{align*}
\tau(t)\ge\tau_0e^{-Ct\|e^{\tau_0\Lambda^\frac{1}{s}}\theta_0\|_{L^2}}.
\end{align*}
Here $C>0$ is a constant which depends on $\nu$ but independent of $t$.
\end{theorem}
\begin{theorem}[Local wellposedness in Sobolev space for the IPMB equations]\label{local-in-time existence theorem for IPMB}
Fix $s>2$ and assume that $\theta_0\in H^s(\mathbb{T}^2)$ has zero-mean on $\mathbb{T}^2$. Then there exists a $T>0$ and a unique smooth solution $\theta^0$ to \eqref{IPMB} with $\nu=0$ such that
$$\theta^0\in L^\infty(0,T;H^s(\mathbb{T}^2)).$$
\end{theorem}
\begin{theorem}[Convergence of solutions as $\nu\rightarrow0$ for the IPMB equations]\label{Convergence of solutions as nu goes to 0 IPMB}
Assume that the hypotheses and notations of Theorem~\ref{local-in-time existence theorem for IPMB} are in force. For $t\in[0,T]$, we have
\begin{align*}
\lim_{\nu\rightarrow0}\|(\theta^{\nu}-\theta^0)(\cdot,t)\|_{H^{s-1}}=0.
\end{align*}
\end{theorem}
\begin{remark}
The results given in Theorem~\ref{local-in-time existence theorem for IPMB} are consistent with those discussed in \cite{CFG11}-\cite{CGO07}. Furthermore , the abstract Theorem~\ref{Local wellposedness in Sobolev space} can also be applied to the non-diffusive SQG equation to show local wellposedness in Sobolev spaces \cite{R95}.
\end{remark}
\begin{remark}
In \cite{FFWV12}, the authors studied the singular incompressible porous media (SIPM) equations set in $\mathbb{T}^2\times[0,\infty)$ with periodic boundary conditions, which are given by
\begin{align}\label{SIPM 1}
&\partial_t\theta+v\cdot\nabla\theta=0,\\
&u=-\nabla(-\Delta)^{-1}\partial_{x_2}\Lambda^\beta\theta-(0,\Lambda^\beta\theta)=M^\beta[\theta].\label{SIPM 2}
\end{align}
The operator $M^\beta$ in \eqref{SIPM 2} is a pseudodifferential operator of order $\beta$, in which the Fourier multiplier symbol can be computed explicitly as $k_1k^{\perp}|k|^{\beta-2}$. It is proved in \cite{FFWV12} that when $0< \beta\le1$ the SIPM equations are ill-posed in Sobolev spaces, however local well-posedness holds for certain patch type weak solutions.

It is straightforward to see that for the case $0<\beta\le1$, the system \eqref{SIPM 1}-\eqref{SIPM 2} satisfies the properties A1--A2 and A5$_1$ (by taking $\nu=0$), so the abstract Theorem~\ref{Gevrey-class local wellposedness} also holds in analogy with those for the MG equations. More specifically, we obtain the following local-in-time Gevrey class existence theorem for the SIPM equations:
\begin{theorem}[Gevrey-class local wellposedness for the SIPM equations]\label{Gevrey-class local wellposedness SIPM}
Fix $\beta\in(0,1]$, $s\ge1$, $r>\frac{5}{2}$ and $K_0>0$. Let $\theta(x,0)=\theta_0$ be of Gevrey-class $s$ with radius of convergence $\tau_0>0$ and satisfies
\begin{align*}
\|\Lambda^re^{\tau_0\Lambda^\frac{1}{s}}\theta(\cdot,0)\|_{L^2}\le K_0.
\end{align*}
There exists $\bar{T},\bar{\tau}>0$ and a unique Gevrey-class $s$ solution $\theta$ to \eqref{SIPM 1}-\eqref{SIPM 2} defined on $\mathbb{T}^2\times[0,\bar{T}]$ with radius of convergence at least $\bar{\tau}$. In particular, there exists a constant $C=C(K_0)>0$ such that for all $t\in[0,\bar{T}]$,
\begin{align*}
\|\Lambda^re^{\bar{\tau}\Lambda^\frac{1}{s}}\theta(\cdot,t)\|_{L^2}\le C.
\end{align*}
\end{theorem} 
\end{remark}
%

\subsection*{Acknowledgment} S. Friedlander is supported by NSF DMS-1613135 and A. Suen is supported by Hong Kong Early Career Scheme (ECS) grant project number 28300016.


\begin{thebibliography}{00}

\bibitem{AB11} J. Azzam and J. Bedrossian, {\it Bounded Mean Oscillation and the Uniqueness of Active Scalar Equations}, Transactions of the American Mathematical Society, Vol 367, No. 5 (2015), 3095--3118.

\bibitem{BCD11} H. Bahouri, J. Chemin and R. Danchin, {\it Fourier Analysis and Nonlinear Partial Differential Equations}, Grundlehren der Mathematischen Wissenschaften 343 (Springer, 2011).

\bibitem{BK12} F. Bernicot and S. Keraani, {\it On the global wellposedness of the 2D Euler equation for a large class of Yudovich type data}, Annales scientifiques de l'ENS 47, fascicule 3 (2014), 559--576.

\bibitem{B49} H. C. Brinkman, {\it A calculation of the viscous force exerted by a flowing fluid on a dense swarm of particles}, Appl. Sci. Res. (1949) 1: 27.

\bibitem{CMT94} P. Constantin, A. Majda and E. Tabak, {\it Formation of strong fronts in the 2D quasi-geostrophic thermal
active scalar}, Nonlinearity, 7 (1994), 1495--1533.

\bibitem{CFG11}  D. C\'{o}rdoba, D. Faraco, and F. Gancedo, {\it Lack of uniqueness for weak solutions of the incompressible porous media equation}, Arch. Ration. Mech. Anal. 200 (2011), no. 3, 725--746.

\bibitem{CGO07} D. C\'{o}rdoba , F. Gancedo, R. Orive, {\it Analytical behavior of two-dimensional incompressible flow in porous media}, J. Math. Phys. 48(6) 065206, 19 (2007).

\bibitem{CV10} L. Caffarelli, A. Vasseur, {\it Drift diffusion equations with fractional diffusion and the quasi-geostrophic equation}, Annals of Mathematics 171 (3), pp. 1903--1930, 2010.

\bibitem{CW09} P. Constantin, J. Wu, {\it H\"{o}lder continuity of solutions of supercritical dissipative hydrodynamic transport equations}, Ann. Inst. H. Poincar\'{e} Anal. Non Linéaire 26 (1), pp. 159--180, 2009.

\bibitem{FFWV12} S. Friedlander, F. Gancedo, W. Sun, and Vlad Vicol, {\it On a singular incompressible porous media equation}, J. Math. Phys. 53, 115602 (2012).

\bibitem{FRV12} S. Friedlander, W. Rusin and V. Vicol, {\it On the supercritically diffusive magnetogeostrophic equations}, Nonlinearity. Volume 25, Number 11 (2012), 3071--3097.

\bibitem{FRV14} S. Friedlander, W. Rusin and V. Vicol, {\it The magnetogeostrophic equations: a survey}, Proceedings of the St. Petersburg Mathematical Society, Volume XV: Advances in Mathematical Analysis of Partial Differential Equations. D. Apushkinskaya, and A.I. Nazarov, eds. pp. 53--78, AMS Translations, Vol 232, 2014.

\bibitem{FS15} S. Friedlander and A. Suen, {\it Existence, uniqueness, regularity and instability results for the viscous magnetogeostrophic equation}, Nonlinearity, 28 (9), 3193--3217, 2015.

\bibitem{FS18} S. Friedlander and A. Suen, {\it Solutions to a class of forced drift-diffusion equations with applications to the magnetogeostrophic equations}, {\it Annals of PDE}, 4(2):14, 2018.

\bibitem{FV11a} S. Friedlander and V. Vicol, {\it Global well-posedness for an advection-diffusion equation arising in magnetogeostrophic dynamics}, Ann. Inst. H. Poincar\'e Anal. Non Lin\'{e}aire, 28(2): pp. 283--301, 2011.

\bibitem{FV11b} S. Friedlander and V. Vicol, {\it On the ill/wellposedness and nonlinear instability of the magnetogeostrophic equations}, Nonlinearity, 24(11):3019--3042, 2011.

\bibitem{KNV07} A. Kiselev, F. Nazarov, and A. Volberg, {\it Global well-posedness for the critical 2D dissipative quasi-geostrophic equation}, Invent. Math. 167, pp. 445--453, 2007.

\bibitem{KV09} I. Kukavica and V. Vicol, {\it On the radius of analyticity of solutions to the three-dimensional Euler equations}, Proceedings of the American Mathematical Society 137, no. 2, 669--677, 2009.

\bibitem{KVW16} I. Kukavica, V. Vicol, and F. Wang. {\it On the ill-posedness of active scalar equations with odd singular kernels}, New Trends in Differential Equations, Control Theory and Optimization: Proceedings of the Eighth Congress of Romanian Mathematicians 2016., 2016.

\bibitem{L00} P.G. Lemari\'{e}-Rieusset, {\it Une remarque sur l'analyticit\'{e} des solutions milds des \'{e}quations de Navier-Stokes dans $R^3$}, C. R. Acad. Sci. Paris S\'{e}r. I Math. 330, no. 3, 183--186, 2000.

\bibitem{LO97} C.D. Levermore and M. Oliver,{\it Analyticity of solutions for a generalized Euler equation}, J. Differential Equations 133, no. 2, 321--339, 1997.

\bibitem{ML94} H.K. Moffatt and D.E. Loper, {\it The magnetostrophic rise of a buoyant parcel in the earth's core}, Geophysical Journal International, 117(2): pp. 394--402, 1994.

\bibitem{MV11} M. Paicu, V. Vicol. {\it Analyticity and Gevrey-class regularity for the second-grade fluid equations.}
Journal of Mathematical Fluid Mechanics 13 (2011), no. 4, 533--555.

\bibitem{M78} H.K. Moffatt, Magnetic Field Generation in Electrically Conducting Fluids, Cambridge Monographs on Mechanics, Cambridge University Press, 1978.

\bibitem{R95} S. Resnick. Dynamical problems in nonlinear advective partial differential equations. PH.D. Thesis, University of Chicago, 1995.

\bibitem{S70} E.M. Stein, {\it Singular Integrals and Differentiability Properties of Function}, Princeton Univ. Press, 1970.

\bibitem{Z89} W. Ziemer, {\it Weakly differentiable functions}, Springer-Verlag, 1989.

\end{thebibliography}
\end{document}